\documentclass[11pt]{amsart}

\usepackage{newpxtext}
\usepackage{newpxmath}
\usepackage{fullpage}
\usepackage{url}
\usepackage{color}

\numberwithin{equation}{section}

\theoremstyle{plain}
\newtheorem{theorem}{Theorem}[section]

\newtheorem{corollary}[theorem]{Corollary}
\newtheorem{lemma}[theorem]{Lemma}
\newtheorem{proposition}[theorem]{Proposition}

\theoremstyle{definition}

\newtheorem{remark}[theorem]{Remark}
\newtheorem{example}[theorem]{Example}

\newtheorem{problem}[theorem]{Problem}

\newcommand{\inv}{^{-1}}
\newcommand{\ldv}{\backslash}
\newcommand{\rdv}{/}
\newcommand{\id}[1]{\mathrm{id}_{#1}}
\newcommand{\lmlt}[1]{\mathrm{LMlt}(#1)}
\newcommand{\mlt}[1]{\mathrm{Mlt}(#1)}
\newcommand{\dis}[1]{\mathrm{Dis}(#1)}
\newcommand{\disp}[1]{\mathrm{Dis}^+(#1)}
\newcommand{\disn}[1]{\mathrm{Dis}^-(#1)}
\newcommand{\End}[1]{\mathrm{End}(#1)}
\newcommand{\aut}[1]{\mathrm{Aut}(#1)}
\newcommand{\aff}[1]{\mathrm{Aff}(#1)}
\newcommand{\Ext}[1]{\mathrm{Ext}(#1)}

\newcommand{\Z}{\mathbb{Z}}

\newcommand{\tr}[1]{\mathrm{tr}(#1)}
\newcommand{\GL}{\mathrm{GL}}

\newcommand{\nalr}[1]{N_{alr}(#1)}
\newcommand{\setof}[2]{\{#1:#2\}}
\newcommand{\genof}[2]{\langle#1:#2\rangle}

\begin{document}

\title{Involutive latin solutions of the Yang-Baxter equation}

\author{Marco Bonatto}
\author{Michael Kinyon}
\author{David Stanovsk\'{y}}
\author{Petr Vojt\v{e}chovsk\'{y}}

\begin{abstract}
Wolfgang Rump showed that there is a one-to-one correspondence between nondegenerate involutive set-theoretic solutions of the Yang-Baxter equation and binary algebras in which all left translations $L_x$ are bijections, the squaring map is a bijection, and the identity $(xy)(xz) = (yx)(yz)$ holds. We call these algebras \emph{rumples} in analogy with quandles, another class of binary algebras giving solutions of the Yang-Baxter equation. We focus on latin rumples, that is, on rumples in which all right translations are bijections as well.

We prove that an affine latin rumple of order $n$ exists if and only if $n=p_1^{p_1 k_1}\cdots p_m^{p_m k_m}$ for some distinct primes $p_i$ and positive integers $k_i$. A large class of affine solutions is obtained from nonsingular near-circulant matrices $A$, $B$ satisfying $[A,B]=A^2$. We characterize affine latin rumples as those latin rumples for which the displacement group generated by $L_x L_y\inv$ is abelian and normal in the group generated by all translations.

We develop the extension theory of rumples sufficiently to obtain examples of latin rumples that are not affine, not even isotopic to a group. Finally, we investigate latin rumples in which the dual identity $(zx)(yx) = (zy)(xy)$ holds as well, and we show, among other results, that the generators $L_x L_y\inv$ of their displacement group have order dividing four.
\end{abstract}

\thanks{M. Kinyon partially supported by Simons Foundation Collaboration Grant 359872. D. Stanovsk\'y partially supported by GA\v CR grant 18-20123S. P. Vojt\v echovsk\'y partially supported by 2019 PROF grant of the University of Denver.}

\address[Bonatto]{Mathematics Research Institute Luis A. Santal\'{o} (IMAS)\\
Universidad de Buenos Aires \\ Buenos Aires, Argentina}

\email{marco.bonatto.87@gmail.com}

\address[Kinyon, Vojt\v{e}chovsk\'{y}]{Department of Mathematics \\
University of Denver \\ Denver, Colorado 80208, USA}

\email{mkinyon@du.edu}
\email{petr@math.du.edu}

\address[Stanovsk\'{y}]{Department of Algebra \\ Faculty of Mathematics and Physics \\ Charles University \\
Prague, Czech Republic}

\email{stanovsk@karlin.mff.cuni.cz}


\date{\today}

\maketitle

\section{Introduction}
\label{sec:introduction}

The quantum Yang-Baxter equation is one of the fundamental equations of mathematical physics. A \emph{set-theoretic solution} of the Yang-Baxter equation over a set $X$ is a mapping $r:X\times X\to X\times X$ such that
\begin{equation}\label{Eq:YB}
    (r\times 1)(1\times r)(r\times 1) = (1\times r)(r\times 1)(1\times r) \tag{YB}
\end{equation}
holds as an equality of mappings $X\times X\times X\to X\times X\times X$. The study of set-theoretic solutions of \eqref{Eq:YB} was initiated by Drinfeld \cite{Dri} and it has resulted in a rich line of research devoted to the existence and classification of set-theoretic solutions of various kinds.

The space of set-theoretic solutions is vast, containing classical algebraic structures such as monoids, distributive lattices and certain self-distributive structures, as well as classes of algebras that have only recently begun to receive attention.

A set-theoretic solution $r=(r_1,r_2)$ of \eqref{Eq:YB} is
\begin{itemize}
  \item \emph{left nondegenerate} if for each $x\in X$, the mapping $y\mapsto r_1(x,y)$ is a permutation of $X$;
  \item \emph{right nondegenerate} if for each $y\in X$, the mapping $x\mapsto r_2(x,y)$ is a permutation of $X$;
  \item \emph{nondegenerate} if $r$ is both left and right nondegenerate;
  \item \emph{bijective} if $r$ is a permutation of $X\times X$;
  \item \emph{involutive} if $r^2=\mathrm{id}_{X\times X}$.
\end{itemize}
Bijective nondegenerate solutions correspond to biracks \cite{FJK,EN}, while involutive nondegenerate solutions
correspond to nondegenerate cycle sets \cite{Rump}.

An algebraic definition of a nondegenerate cycle set can be given as follows. A \emph{left quasigroup} is a binary algebra $(X,\cdot)$ in which all left translations $L_x:y\mapsto xy$ are bijections of $X$. A \emph{cycle set} is then a left quasigroup $(X,\cdot)$ in which the identity
\begin{equation}\label{Eq:LeftRump}
    (x\cdot y)\cdot (x\cdot z) = (y\cdot x)\cdot (y\cdot z)         \tag{R$_\ell$}
\end{equation}
holds. This can also be conveniently expressed using left translations, namely as
\begin{equation}\label{Eq:LeftRumpTrans}
  L_{x\cdot y} L_x = L_{y\cdot x} L_y\,.      \tag{R$_{\ell}'$}
\end{equation}
A binary algebra $(X,\cdot)$ is \emph{uniquely $2$-divisible} if the squaring map
\[
    \sigma:X\to X;\quad x\mapsto x\cdot x = x^2
\]
is a bijection of $X$. A cycle set $X$ is \emph{nondegenerate} if it is uniquely $2$-divisible.

We propose to rename nondegenerate cycle sets as \emph{rumples}, both to acknowledge Rump's contributions and to highlight the similarity of rumples to quandles. Thus, a \emph{rumple} is a uniquely $2$-divisible left quasigroup satisfying \eqref{Eq:LeftRump}.

Several structures, algebraic or otherwise, have been developed to construct and classify solutions of \eqref{Eq:YB}. For example, bijective 1-cocycles \cite{ESS}, I-type structures \cite{CG,GB}, cycle sets \cite{Deh,Rump,Ven} and braces \cite{CJO,GI-braces,Rump-braces} all stem from the study of involutive, nondegenerate solutions. Braces have been generalized to skew-braces \cite{GV} for bijective, nondegenerate solutions. Skew braces have been generalized to semi-braces \cite{CCS} for left nondegenerate solutions.

Many rumples of a combinatorial flavor are obtained from so-called multipermutational solutions of \eqref{Eq:YB}; they include the $2$-reductive medial quandles studied in \cite{JPSZ}. We are more interested in rumples that are algebraically connected or, even more strongly, that are quasigroups. Since the multiplication tables of finite quasigroups are precisely latin squares, it is customary to designate quasigroups within various classes of algebras by the adjective \emph{latin}, cf. latin quandles. The main results of this paper are concerned with \emph{latin rumples}.

We conclude this introduction with a summary of the paper. In {\S}\ref{sec:rumples}, we introduce additional notation and terminology, and besides adumbrating Rump's basic results \cite{Rump} in our preferred notation and terminology, we also discuss how rumples interact with other kinds of set-theoretic solutions of \eqref{Eq:YB}, such as biracks, racks, biquandles and quandles.
In the brief {\S}\ref{sec:latin}, we build upon Rump's results and show that there is a one-to-one correspondence between latin rumples and involutive, nondegenerate solutions $r=(r_1,r_2)$ of \eqref{Eq:YB} in which both $r_1$ and $r_2$ are quasigroups.

In {\S}\ref{sec:affine}, we give a thorough study of affine latin rumples. We answer the question for which finite orders $n$ there exist affine latin rumples (see Theorem \ref{Th:Spectrum}), we obtain a class of latin rumples from matrices $A$, $B\in\GL_p(p)$ that are close to circulant matrices and satisfy $[A,B]=A^2$ or equivalently, $[B,A\inv]=I$. This
last equation is the Heisenberg commutation relation, and so finding solutions of \eqref{Eq:YB} based on such matrices is essentially the same as classifying finite dimensional modules of the first Weyl algebra over finite fields with invertible generators (\cite{Lam}, p.7). We do not pursue this connection any further here, but consider it to be an interesting possible future direction for the study of affine rumples. We conclude the section by paying close attention to the displacement group and using it to characterize affine latin rumples within the class of all latin rumples (see Theorem \ref{Th:affine_char}).

In {\S}\ref{sec:rlinear}, we study latin rumples isotopic to groups (a class that properly contains affine latin rumples) and we again characterize them in terms of their displacement groups (see Theorem \ref{Th:right linear}). In {\S}\ref{sec:nilpotent} we develop the theory of central extensions of latin rumples and we construct latin rumples that are not affine, nor even isotopic to a group. Finally, in {\S}\ref{sec:both-sided} we study latin rumples which satisfy not only the identity \eqref{Eq:LeftRump} but also its mirror image
\begin{equation}\label{Eq:RightRump}
    (z\cdot x)\cdot (y\cdot x) = (z\cdot y)\cdot(x\cdot y)\,,          \tag{R$_r$}
\end{equation}
or equivalently,
\begin{equation}\label{Eq:RightRumpTrans}
    R_{y\cdot x} R_x = R_{x\cdot y} R_y\,.      \tag{R$_r'$}
\end{equation}

\section{Rumples}
\label{sec:rumples}

\subsection{Quasigroup properties}
\label{subsec:quasigroups}

In a left quasigroup $(X,\cdot)$, we denote by $x\ldv y$ the unique solution $u\in X$ to the equation $x\cdot u=v$, and refer to the binary operation $\ldv$ as \emph{left division}. Then
\begin{equation}\label{Eq:LeftQuasigroup}
    x\cdot (x\ldv y) = y = x\ldv(x\cdot y)
\end{equation}
holds for every $x$, $y\in Q$. Conversely, any algebra $(X,\cdot,\ldv)$ satisfying \eqref{Eq:LeftQuasigroup} is a left quasigroup with left division $\ldv$. A homomorphism of left quasigroups $(X_1,\cdot_1,\ldv_1)\to (X_2,\cdot_2,\ldv_2)$ is a mapping $f:X_1\to X_2$ satisfying $f(x\cdot_1 y) = f(x)\cdot_2 f(y)$ for every $x$, $y\in X$. It then follows that $f(x\ldv_1 y) = f(x)\ldv_2 f(y)$, too.

Dually, a \emph{right quasigroup} is a binary algebra $(X,\cdot)$ in which all right translations $R_x:y\to yx$ are bijections of $X$. Then the unique solution $v\in X$ to $v\cdot x=y$ will be denoted by $y\rdv x$. Right division satisfies the identities $(x\cdot y)/y = x = (x/y)\cdot y$. A \emph{quasigroup} is a left quasigroup that is also a right quasigroup.

We adopt the following notational convention for quasigroups. The multiplication operation will be denoted by both juxtapositon and by $\cdot$. The $\cdot$ multiplication is less binding than the division operations, which are in turn less binding than juxtapositon. For instance, $x/yz\cdot uv$ abbreviates $(x/(y\cdot z))\cdot (u\cdot v)$.

The \emph{left multiplication group} of a left quasigroup $X$ is the permutation group generated by all left translations, i.e.,
\[
\lmlt X=\genof{L_x}{x\in X}\,.
\]
If $X$ is a quasigroup, we also define the \emph{multiplication group} as the permutation group generated by all left and right translations, i.e.,
\[
\mlt X =\genof{L_x,R_x}{x\in X}\,.
\]

Two binary algebras $(X_1,\cdot_1)$, $(X_2,\cdot_2)$ are \emph{isotopic} if there are bijections $f$, $g$, $h:X_1\to X_2$ such that $f(x)\cdot_2 g(y) = h(x\cdot_1 y)$ holds for all $x$, $y\in X_1$.

\subsection{Rump left quasigroups and rumples}
\label{subsec:rump}

A left quasigroup satisfying \eqref{Eq:LeftRump} will be called a \emph{Rump left quasigroup}. Thus Rump
left quasigroups are the \emph{cycle sets} of \cite{Rump}, and also the \emph{RC quasigroups} of \cite{Deh}
(but note that RC quasigroups need not be quasigroups).

If $(X,\cdot)$ is a uniquely $2$-divisible binary algebra, then for every $x\in X$ there exists a unique element
$x^{1/2}\in X$, the \emph{square root of $x$}, such that $x^{1/2}x^{1/2}=x$. As already mentioned in {\S}\ref{sec:introduction}, we define a \emph{rumple} to be a uniquely $2$-divisible, Rump left quasigroup
(i.e., a nondegenerate cycle set in Rump's own terminology).

Rump proved that, in our terminology, a finite Rump left quasigroup is a rumple, that is, is uniquely $2$-divisible \cite[Thm.~2]{Rump}. Rump's proof, though short on its own, uses deep structure theory. Here we give a short combinatorial proof that uses nothing more than the left Rump identity \eqref{Eq:LeftRump}.

\begin{theorem}\label{thm:torsion}
  Let $(X,\cdot)$ be a Rump left quasigroup such that $\lmlt{X}$ is a torsion group. Then the squaring map
  $\sigma:X\to X$ is surjective.
\end{theorem}
\begin{proof}
Let $(X,\cdot)$ be a Rump left quasigroup and fix $c\in X$. Define a sequence $(c_n)_{n\geq 0}$ by setting $c_0 = c$ and $c_n = (c_{n-1}\ldv c)c_{n-1}$ for $n\ge 1$. Then
\[
    c_n^2 = (c_{n-1}\ldv c)c_{n-1}\cdot (c_{n-1}\ldv c)c_{n-1} = c_{n-1}(c_{n-1}\ldv c)\cdot c_{n-1}c_{n-1} = L_c(c_{n-1}^2)
\]
for every $n\ge 1$, using \eqref{Eq:LeftRump} in the second equality. By induction, we have $c_n^2 = L_c^{n+1}(c)$ for every $n\ge 0$. Since $\lmlt{X}$ is a torsion group, there exists $n\geq 0$ such that $L_c^{n+1} = \id{X}$. Then $\sigma(c_n) = c_n^2 = L_c^{n+1}(c)=c$.
\end{proof}

\begin{corollary}[Rump {\cite[Thm.~2]{Rump}}]
   Every finite Rump left quasigroup is a rumple.
\end{corollary}
%

The number of rumples up to isomorphism has been recorded for small orders in the On-Line Encyclopedia of Integer Sequences \cite{Sloane} as sequence A290887:
\begin{center}
\begin{tabular}{rrrrrrrrr}
order\qquad\qquad & 1 & 2 & 3 & 4 & 5 & 6 & 7 & 8 \\ \hline
number of rumples\qquad\qquad & 1 & 2 & 5 & 23 & 88 & 595 & 3456 & 34528
\end{tabular}
\end{center}

\subsection{Rumples and the Yang-Baxter equation}
\label{sbsec:rump_yb}

The left component function $r_1$ of a left nondegenerate solution
$r = (r_1,r_2)$ of \eqref{Eq:YB} is a left quasigroup and thus has its own left division operation. It turns
out to be useful to view $r_1$ itself as the left division operation $r_1(x,y) = x\ldv y$ of a left quasigroup
$(X,\cdot,\ldv)$. Put another way, it is more convenient to work with the operation $\cdot$ defined by
$x\cdot y = z$ if and only if $r_1(x,z) = y$ instead of $r_1(x,y) = z$. In the special case of involutive left nondegenerate solutions, the right component function $r_2$ must also have a specific form.

\begin{lemma}\label{lem:involutive}
    Let $(X,\cdot,\ldv)$ be a left quasigroup. Then the mapping $r: X\times X\to X\times X$ defined by
    $r(x,y)= (x\ldv y,r_2(x,y))$ is involutive if and only if $r_2(x,y) = (x\ldv y)x$ for all $x,y\in X$.
\end{lemma}
\begin{proof}
  If $r^2 = \id{X\times X}$, then $(x\ldv y)\ldv r_2(x,y) = x$ and so $r_2(x,y) = (x\ldv y)x$. Conversely,
  if $r_2(x,y) = (x\ldv y)x$, then it is straightforward to check that $r^2 = \id{X\times X}$.
\end{proof}

The following result explains why one is led naturally to the left Rump identity \eqref{Eq:LeftRump} from
set-theoretic solutions of \eqref{Eq:YB}.

\begin{theorem}[Rump {\cite[Prop.~1]{Rump}}]\label{Th:YBE-rumplq}
There is a one-to-one correspondence between Rump left quasigroups and involutive left nondegenerate
solutions of the Yang-Baxter equation.
\begin{enumerate}
  \item If $(X,\cdot)$ is a Rump left quasigroup, then $r(x,y) = (x\ldv y, (x\ldv y)x)$ is an involutive left nondegenerate solution of \eqref{Eq:YB}.
  \item If $r(x,y) = (r_1(x,y),r_2(x,y))$ is an involutive left nondegenerate solution of \eqref{Eq:YB}, then the operation $\cdot$ given by $x\cdot y=z\iff r_1(x,z) = y$ defines a Rump left quasigroup $(X,\cdot)$.
\end{enumerate}
\end{theorem}
%

The correspondence of Theorem \ref{Th:YBE-rumplq} restricts to rumples and nondegenerate solutions.

\begin{theorem}[Rump {\cite[Props.~1 and 2]{Rump}}]\label{thm:YBE-rumples}
There is a one-to-one correspondence between rumples and involutive nondegenerate solutions of the Yang-Baxter equation.
\begin{enumerate}
  \item If $(X,\cdot)$ is a rumple, then $r(x,y) = (x\ldv y, (x\ldv y)x)$ is an involutive nondegenerate solution of \eqref{Eq:YB}.
  \item If $r(x,y) = (r_1(x,y),r_2(x,y))$ is an involutive nondegenerate solution of \eqref{Eq:YB}, then the operation $\cdot$ given by $x\cdot y=z\iff r_1(x,z) = y$ defines a rumple $(X,\cdot)$.
\end{enumerate}
\end{theorem}




Note that involutive solutions are obviously bijective. Bijective nondegenerate solutions of the Yang-Baxter equation are called \emph{biracks}. Biracks can be used to construct coloring invariants of knots and links \cite[Chapter 5]{EN}. Invariance with respect to the 3rd Reidemeister move is equivalent to the Yang-Baxter equation, while the invariance with respect to the 2nd Reidemeister move is ensured by bijectivity and nondegeneracy. To achieve invariance with respect to the 1st Reidemeister move, it suffices to impose
the condition
\begin{equation}\label{Eq:Biracks-LV}
    \text{there is a permutation $t$ of $X$ such that $r(t(x),x)=(t(x),x)$},
\end{equation}
cf. \cite{LV2}. A \emph{biquandle} is a birack satisfying \eqref{Eq:Biracks-LV}. See \cite{EN} or \cite{FJK} for an alternative axiomatization of biracks and biquandles based on exchange laws.

Via the correspondence of Theorem \ref{thm:YBE-rumples}, rumples form a subclass of biquandles:

\begin{proposition}\label{Pr:biquandle}
Let $X$ be a rumple and let $r(x,y) = (x\ldv y,(x\ldv y)x)$ be the corresponding nondegenerate involutive solution.
Then $(X,r)$ is a biquandle.
\end{proposition}
\begin{proof}
It remains to verify the condition \eqref{Eq:Biracks-LV}. Let $t(x)=\sigma\inv(x)=x^{1/2}$. Then $r(t(x),x) = r(x^{1/2},x)= (x^{1/2}\ldv x,(x^{1/2}\ldv x)x^{1/2})=(x^{1/2},x) = (t(x),x)$.
\end{proof}

A \emph{rack} is a birack $r=(r_1,r_2)$ satisfying $r_2(x,y)=x$. Algebraically,
a rack is a left quasigroup satisfying the left self-distributive law
\[
    (xy)(xz) = x(yz)\,.
\]
We point out that if a left quasigroup $(X,\cdot)$ is a rack with left division $\ldv$, then $(X,\ldv)$ is also
a rack and conversely. Hence the correspondence between racks and bijective nondegenerate solutions with
$r_2(x,y) = x$ can be stated in terms of left division operations, analogously to the correspondence in
Theorem \ref{thm:YBE-rumples}. Note that for racks, the condition \eqref{Eq:Biracks-LV} is equivalent to idempotence $xx=x$.
Idempotent racks are known as \emph{quandles} \cite{J,M}.

The definitions of racks and Rump left quasigroups are syntactically very similar but they behave quite differently
as algebraic structures.

The analogy between quandles and rumples can be further strengthened by the following compilation of two results in
the literature; see Stein \cite{Ste} for finite latin quandles and Etingof, Schedler and Soloviev \cite[Theorem 2.15]{ESS}
for finite rumples.

\begin{proposition}\label{Pr:solvable}
  If $X$ is a finite latin quandle or a finite rumple, then the group $\lmlt{X}$ is solvable.
\end{proposition}

\subsection{Intersection of rumples and quandles}
\label{sbsec:rump_rack}

There is a class of natural examples in the intersection of quandles and rumples:

\begin{example}
The conjugation quandle over a group $G$ satisfies \eqref{Eq:LeftRump} if and only if $G$ is nilpotent of class 2.
\end{example}

It is easy to characterize the intersection of rumples and racks. A left quasigroup is called \emph{2-reductive} if it satisfies the identity $(xy)z=yz$.
Expressing the Rump identity by \eqref{Eq:LeftRumpTrans}, i.e., $L_{xy}L_x=L_{yx}L_y$, left distributivity by $L_{xy}L_x=L_xL_y$, and 2-reductivity by $L_{xy}=L_y$, we immediately obtain:

\begin{proposition}
For a left quasigroup, any two of the following three conditions imply the third:
\begin{itemize}
\item left distributivity;
\item left Rump identity;
\item 2-reductivity.
\end{itemize}
The intersection of the classes of Rump left quasigroups and racks is the class of 2-reductive racks.
\end{proposition}

In the context of the Yang-Baxter equation, the intersection of rumples and racks corresponds to multipermutational
solutions of level $2$ with $r_2(x,y)=x$. Following \cite[{\S}3.2]{ESS}, a solution $(X,r)$ is called \emph{multipermutational of level $n$} if the $n$-th retract $\mathrm{Ret}^n(X,r)$ is trivial. (Level 2 has been
studied extensively in \cite{GS-multiperm2,JPZ-multiperm2}). A rack is multipermutational of level 2 if and only if $L_{yx}=L_{zx}$ for every $x$, $y$, $z$, which is in turn equivalent to 2-reductivity, since $L_{yx}=L_{xx}=L_x$.

The intersection of rumples and quandles is the class of 2-reductive quandles which was studied in \cite[{\S}6, 8]{JPSZ}, where a general construction was given and $2$-reductive quandles were counted up to isomorphism for all orders up to $16$.

See \cite{L} on the interplay between self-distributivity and other types of solutions to the Yang-Baxter equation.

\subsection{$\Delta$-bijectivity}

A binary algebra $(X,\cdot)$ is said to be \emph{$\Delta$-bijective} if the mapping
\[
    \Delta_{(\cdot)} : X\times X\to X\times X,\quad (x,y)\mapsto (xy,yx)
\]
is bijective. In this subsection we show that a Rump left quasigroup is $\Delta$-bijective if and only if it is a rumple. The idea comes from \cite{Rump} but our proofs are different.

\begin{lemma}\label{Lm:Delta}
  Let $(X,\cdot)$ be a $\Delta$-bijective binary algebra. Then $\Delta_{(\cdot)}\inv = \Delta_{(\ast)}$
  for some binary operation $\ast$ on $X$.
\end{lemma}
\begin{proof}
Write $\Delta\inv_{(\cdot)}(x,y) = (x\ast y,x\diamond y)$. The equation $\Delta_{(\cdot)}\Delta\inv_{(\cdot)} = \id{X\times X}$ then says
\begin{equation}
    (x\ast y)(x\diamond y) = x\qquad\text{and}\qquad (x\diamond y)(x\ast y) = y\,, \label{Eq:Delta1}
\end{equation}
while the first component of $\Delta\inv_{(\cdot)}\Delta_{(\cdot)} = \id{X\times X}$ yields $(xy)\ast (yx) = x$. Replacing $x$ with $x\diamond y$ and $y$ with $x\ast y$, we get
\[
x\diamond y = [(x\diamond y)(x\ast y)]\ast [(x\ast y)(x\diamond y)] = y\ast x\,,
\]
using \eqref{Eq:Delta1} in the second equality.
\end{proof}

\begin{lemma}\label{Lm:OneImplication}
Every $\Delta$-bijective binary algebra is uniquely $2$-divisible.
\end{lemma}
\begin{proof}
Let $\ast$ be the binary operation on $X$ such that $\Delta_{(\ast)} = \Delta_{(\cdot)}\inv$ (by Lemma \ref{Lm:Delta}).
The components of the equations $\Delta_{(\cdot)}\Delta_{(\ast)}(x,x) = (x,x)$ and $\Delta_{(\ast)}\Delta_{(\cdot)}(x,x) = (x,x)$ give $(x\ast x)(x\ast x) = x$ and $(xx)\ast (xx) = x$.
Thus $x\ast x$ is the unique square root of $x$ in $(X,\cdot)$.
\end{proof}

The converse implication of Lemma \ref{Lm:OneImplication} is not true for general binary algebras, as witnessed by a nontrivial cyclic group of odd order.

\begin{lemma}\label{Lm:Bijective}
Every rumple $(X,\cdot)$ is $\Delta$-bijective and
\begin{displaymath}
    \Delta_{(\cdot)}\inv (x,y) = ( (x\ldv y^2)^{1/2}, (y\ldv x^2)^{1/2} )
\end{displaymath}
holds for all $x,y\in X$.
\end{lemma}

\begin{proof}
Set $x\ast y = (x\ldv y^2)^{1/2}$. We will show that $\Delta\inv_{(\cdot)} = \Delta_{(\ast)}$. Consider the identity $yx\cdot yx = xy\cdot xx$, a consequence of \eqref{Eq:LeftRump}. This is equivalent to
$(xy)\ldv (yx)^2 = x^2$, and then taking square roots, we have $(xy)\ast (yx) = x$. Reversing the roles of
$x$ and $y$, we also have $(yx)\ast (xy) = y$. This establishes $\Delta_{(\ast)} \Delta_{(\cdot)} = \id{X\times X}$.

Next set $u = (x\ldv y^2)^{1/2}$. Then
\begin{equation}\label{Eq:Del_tmp}
(u\ldv x)u\cdot (u\ldv x)z = u(u\ldv x)\cdot uz = x\cdot uz\,,
\end{equation}
using \eqref{Eq:LeftRump}. Taking $z = u$, we have $[(u\ldv x)u]^2 = x\cdot u^2 = y^2$ and so $(u\ldv x)u = y$. Using this in \eqref{Eq:Del_tmp}, we have $y\cdot (u\ldv x)z = x\cdot uz$. Setting
$z = u\ldv x$, we get $y\cdot (u\ldv x)^2 = x^2$, and so $(u\ldv x)^2 = y\ldv x^2$. Taking square roots and
then multiplying on the left by $u$, we obtain $(x\ldv y^2)^{1/2} (y\ldv x^2)^{1/2} = x$.
Reversing the roles of $x$ and $y$, we also have $(y\ldv x^2)^{1/2} (x\ldv y^2)^{1/2} = y$.
This establishes $\Delta_{(\cdot)} \Delta_{(\ast)} = \id{X\times X}$.
\end{proof}

Combining Lemmas \ref{Lm:OneImplication} and \ref{Lm:Bijective}, we obtain:

\begin{proposition}
A Rump left quasigroup is $\Delta$-bijective if and only if it is a rumple.
\end{proposition}

\begin{remark}\label{Rem:u2d}
Let $(X,\cdot)$ be a rumple. Motivated by the particular form of $\Delta_{(\cdot)}\inv$ in Lemma \ref{Lm:Bijective}, define $X^\partial = (X,*)$ by
\[
x\ast y = (x\ldv y^2)^{1/2}\,.
\]
Then $X^{\partial}$ is a rumple, called the \emph{dual rumple} of $X$. The left division in $X^{\partial}$ is $x\ldv^* y = (x y^2)^{1/2}$ and the unique square root of $x$ in $X^{\partial}$ is $xx$. If $X$ is latin, then so is $X^{\partial}$ with $x\rdv^* y = y^2\rdv x^2$. Finally, $(X^{\partial})^{\partial} = X$ but the two rumples $X$ and $X^\partial$ are not necessarily isomorphic. See \cite{Deh} for more details.
\end{remark}

\subsection{The displacement group}
\label{subsec:displacement}

Displacement groups have proven to be very useful in the theory of quandles (see \cite{BS,HSV}). It will become
apparent that they are also important for rumples.

Let $X$ be a left quasigroup. The \emph{positive displacement group} $\disp{X}$ and the
\emph{negative displacement group} $\disn{X}$ are the subgroups of $\lmlt{X}$ defined, respectively, by
\[
\disp{X} = \genof{L_xL_y\inv}{x,y\in X} \qquad\text{and}\qquad
\disn{X} = \genof{L_x\inv L_y}{x,y\in X}\,.
\]
The \emph{displacement group} $\dis{X}$ is the group
\[
\dis{X} = \genof{L_xL_y\inv, L_x\inv L_y}{x,y\in X}\,.
\]
Note that for a fixed $e\in X$, we have $\disp{X} = \genof{L_x L_e\inv}{x\in X}$ and $\disn{X} = \genof{L_e\inv L_x}{x\in X}$ since $L_x L_y\inv = L_x L_e\inv (L_y L_e\inv)\inv$ and $L_x\inv L_y = (L_e\inv L_x)\inv L_e\inv L_y$.

The Rump identity \eqref{Eq:LeftRumpTrans} can be restated as $L_xL_y\inv =L_{xy}\inv L_{yx}$, hence, for
Rump left quasigroups,
\[
\disp{X}\leq \disn{X} = \dis{X}\,.
\]

\begin{lemma}\label{lem:dis}
Let $X$ be a rack or a rumple. Then $\dis{X}=\disp X=\disn X$.
\end{lemma}
\begin{proof}
The argument is easy for racks. We have $L_xL_y = L_{xy}L_x$ and hence $L_yL_x\inv =L_x\inv L_{xy}$, which
implies $\disp{X}\le\disn{X}$. Also, $L_x\inv L_y = L_x\inv L_{x(x\ldv y)} = L_{x\ldv y}L_x\inv $, which
implies $\disn{X}\le\disp{X}$.

Suppose now that $X$ is a rumple. From the remark preceding the lemma, $\disp{X}\leq \disn{X}$.
Using $\Delta$-bijectivity, for every $a$, $b\in X$, there exist $x$, $y\in X$ such that $xy=a$ and $yx=b$.
Thus $L_a\inv L_b = L_x L_y\inv$ by \eqref{Eq:LeftRumpTrans} and we get $\disn{X}\leq \disp{X}$.
\end{proof}

\begin{problem}
  If $X$ is a Rump left quasigroup, does $\disp{X} = \disn{X}$?
\end{problem}

\begin{proposition}\label{Pr:dis}
Let $X$ be a left quasigroup such that $\dis{X} = \disp{X} = \disn{X}$. Then:
\begin{enumerate}
\item $\dis{X}\unlhd \lmlt{X}$.
\item $\lmlt{X} / \dis{X}$ is a cyclic group.
\item $\dis{X} = \setof{L_{x_1}^{k_1}\dots L_{x_n}^{k_n}}{0\leq n,\,x_i\in X,\,\sum k_i=0}$.
\end{enumerate}
\end{proposition}

\begin{proof}
(1) It is sufficient to prove that every conjugate of $L_xL_y\inv$ by $L_z^{\pm 1}$ is in $\dis{X}$.
Clearly, $L_z\inv L_xL_y\inv L_z\in\dis X$. For the other conjugate, write
$L_x L_y\inv =L_{x_1}\inv L_{y_1}\cdots L_{x_n}\inv L_{y_n}$ for some $x_1,\dots,x_n$, $y_1,\dots,y_n\in X$,
and regroup $L_z L_x L_y\inv L_z\inv = (L_z L_{x_1}\inv)(L_{y_1} L_{x_2}\inv)\cdots(L_{y_n} L_z\inv)\in\dis{X}$.

(2) Fix $e\in X$. For every $x\in X$, we have $L_x\inv L_e\in\dis X$, and thus $L_x\dis{X} = L_e\dis{X}$.
Consequently, for $\alpha = L_{x_1}^{k_1}\dots L_{x_n}^{k_n}\in\lmlt{X}$ we have
$\alpha\dis{X} = L_e^{k_1+\cdots+k_n}\dis{X}$, so $\lmlt{X}/\dis{X} = \langle L_e\dis{X}\rangle$.

(3) Let $S=\setof{L_{x_1}^{k_1}\dots L_{x_n}^{k_n}}{0\le n,\,x_i\in X,\,\sum k_i=0}$. Every $\alpha\in S$ can be written
as $\alpha = L_{x_1}^{k_1}\dots L_{x_n}^{k_n}$, where $k_i=\pm 1$ for every $1\le i\le n$. We prove by induction on $n$
that $\alpha\in \dis{X}$. If $n=2$, we are done by the definition of $\dis{X}$, so suppose that $n>2$. If $k_1=k_n$ then
there must be an $m$ such that $1<m<n$ and $\sum_{i=1}^m k_i=0=\sum_{i=m+1}^n k_i$. By the induction hypothesis, $\alpha$
is then a product of two elements of $\dis{X}$. Finally suppose that $k_1 = -k_n$. Then $\sum_{i=2}^{n-1} k_i=0$ and hence
$\alpha = L_x\beta L_y\inv$ or $\alpha = L_x\inv\beta L_y$ for some $\beta\in\dis{X}$ and some $x,y\in X$. In the former
case, we can write $\beta = L_{u_1}\inv L_{v_1}\cdots L_{u_s}\inv L_{v_s}$ for some $u_i$, $v_i$ and observe that
$\alpha = L_x\beta L_y\inv$ is a product of factors of the form $L_a L_b\inv$, while in the latter case we can write
$\beta = L_{u_1}L_{v_1}\inv\cdots L_{u_s}L_{v_s}\inv$ and observe that $\alpha = L_x\inv\beta L_y$ is a product of
factors of the form $L_a\inv L_b$.
\end{proof}

We deduce the following result for racks and rumples. For racks, this was already known, cf. \cite[Prop.~2.1]{HSV}.

\begin{proposition}
The conditions \emph{(1)--(3)} of Proposition \emph{\ref{Pr:dis}} hold when $X$ is a rack or a rumple.
\end{proposition}

\section{Latin rumples}\label{sec:latin}

Recall that a rumple is a uniquely $2$-divisible left quasigroup satisfying the identity \eqref{Eq:LeftRump}.
A \emph{latin rumple} is a rumple that is a quasigroup. It is not necessary to assume unique $2$-divisibility
in the definition of a latin rumple:

\begin{proposition}\label{Pr:latin_2-div}
A binary algebra $(X,\cdot)$ is a latin rumple if and only if it is a quasigroup satisfying \eqref{Eq:LeftRump}. Furthermore, in a latin rumple $(X,\cdot)$, the squaring map is given by $\sigma = R_{ee}L_eR_{e}\inv$, where $e$ is any
element of $X$.
\end{proposition}
\begin{proof}
Suppose that $(X,\cdot)$ is a quasigroup satisfying \eqref{Eq:LeftRump} and let $e\in X$. The bijection
$\sigma=R_{ee}L_e R_{e}\inv$ then satisfies
\[
    \sigma(x) = e(x\rdv e)\cdot ee = (x\rdv e)e\cdot (x\rdv e)e = x^2,
\]
where we have used \eqref{Eq:LeftRump} in the second step.
\end{proof}


Latin rumples form a very natural class of set-theoretic solutions of the Yang-Baxter equation.

\begin{theorem}\label{Th:latin}
There is a one-to-one correspondence between latin rumples and involutive solutions $r=(r_1,r_2)$ of the
Yang-Baxter equation in which both $r_1$, $r_2$ are quasigroup operations.
\end{theorem}
\begin{proof}
Let $(r_1,r_2):X\times X\to X\times X$ be an involutive solution of the Yang-Baxter equation in which both
$r_1$, $r_2$ are quasigroup operations. By Theorem \ref{thm:YBE-rumples}, the operation $\cdot$ defined by
$xy=z\iff r_1(x,z)=y$ defines a rumple $(X,\cdot)$. Since $r_1$ is a quasigroup (not just a left quasigroup),
$(X,\cdot)$ is latin.

Conversely, if $(X,\cdot)$ is a latin rumple then Theorem \ref{thm:YBE-rumples} shows that $r=(r_1,r_2)$
with $r_1(x,y)=x\ldv y$ and $r_2(x,y)=(x\ldv y)x$ is an involutive nondegenerate solution, and so
$(X,r_1) = (X,\ldv)$ is a left quasigroup and $(X,r_2)$ is a right quasigroup. Since $(X,\cdot)$ is a quasigroup,
$(X,r_1)$ is also a quasigroup. To show that $r_2$ is a left quasigroup, we note that the equation
$(x\ldv y)x = z$ has a unique solution $y$ in $X$, namely $y = x(z\rdv x)$.
\end{proof}

\begin{example}\label{ex:4}
An exhaustive search using the finite model builder \texttt{Mace4} \cite{McCune} reveals that up to isomorphism there are only two nontrivial latin rumples of order less than $12$, namely
\[
    \begin{array}{c|cccc}
         X_{4,1}&0&1&2&3\\
        \hline
        0&0&1&3&2\\
        1&2&3&1&0\\
        2&1&0&2&3\\
        3&3&2&0&1
    \end{array}
    \qquad\text{and}\qquad
    \begin{array}{c|cccc}
        X_{4,2}&0&1&2&3\\
        \hline
        0&1&3&0&2\\
        1&0&2&1&3\\
        2&2&0&3&1\\
        3&3&1&2&0
    \end{array}
   \,.
\]
It turns out that both $X_{4,1}$ and $X_{4,2}$ are self-dual in the sense of Remark \ref{Rem:u2d}, and both satisfy the right Rump identity \eqref{Eq:RightRump}.
\end{example}

We will need additional structure theory to find more latin rumples.

\section{Affine latin rumples}\label{sec:affine}

\def\A{\varphi}
\def\B{\psi}

\subsection{Linear and affine representations}

Let $(G,+)$ be an abelian group with identity element $0$, let $\A$ and $\B$ be endomorphisms of $(G,+)$, and let $c\in G$. Then the binary algebra $(G,*)$ defined by
\begin{equation}\label{Eq:Rep}
    x*y=\A(x)+\B(y)+c
\end{equation}
is called \emph{affine over $(G,+)$}. If $c=0$, it is called \emph{linear over $(G,+)$}. We will denote the algebra $(G,*)$ by $\aff{G,+,\A,\B,c}$ or by $\aff{G,\A,\B,c}$ if the group operation on $G$ is understood from the context.

Note that $\aff{G,+,\A,\B,c}$ is a left quasigroup if and only if $\B\in\aut{G,+}$, and it is latin if and only if $\A$, $\B\in\aut{G,+}$. Also note that \eqref{Eq:Rep} shows that an affine quasigroup $\aff{G,+,\A,\B,c}$ is isotopic to the abelian group $(G,+)$.

An algebra $(G,*)$ is called \emph{affine} (resp. \emph{linear}) if it is affine (resp. linear) over some abelian group $(G,+)$. In the literature, affine quasigroups are also called \emph{central} or \emph{T-quasigroups} \cite{Smi-book,Sta-latin}. (We will resist the urge to use the $T$-terminology for Rump quasigroups.)

Our definitions of linear and affine binary algebras are compatible with the definitions of linear and affine solutions of the Yang-Baxter equation from \cite[Section 3.1]{ESS}. Formally, linear (affine) nondegenerate involutive solutions correspond, in the sense of Theorem \ref{thm:YBE-rumples}, to linear (affine) rumples.

\begin{proposition}\label{Pr:AffineRump}
An affine binary algebra $\aff{G,+,\A,\B,c}$ satisfies \eqref{Eq:LeftRump} if and only if
\begin{equation}\label{Eq:AffineRump}
    [\A,\B] = \A\B-\B\A = \A^2.
\end{equation}
If $\aff{G,+,\A,\B,c}$ is a quasigroup then \eqref{Eq:AffineRump} holds if and only if
\begin{equation}\label{Eq:AffineRump1}
    [\B,\A\inv] = 1,
\end{equation}
which is further equivalent to
\begin{equation}\label{Eq:AffineRump2}
    [\A\inv,\B\inv] = \B^{-2}.
\end{equation}
\end{proposition}

\begin{proof}
Let us write $\A x$ instead of $\A(x)$, etc. For $x$, $y$, $z\in G$ we have
\[
    (x*y)*(x*z) = (\A x+\B y+c)*(\A x +\B z+c) = \A^2x+\A\B y+\A c+\B\A x+\B^2z+\B c+c,
\]
while
\[
    (y*x)*(y*z) = (\A y+\B x+c)*(\A y+\B z+c) = \A^2y+\A\B x+\A c+\B\A y+\B^2z+\B c+c.
\]
Hence the identity $(x*y)*(x*z)=(y*x)*(y*z)$ holds if and only if $\A^2u+\B\A u=\A\B u$ for every $u\in G$. Multiplying $\A\B-\B\A=\A^2$ by $\A\inv$ from both sides, we obtain $\B\A\inv -\A\inv \B=1$. Multiplying further by $\B\inv$ from both sides, we obtain $\A\inv\B\inv -\B\inv\A\inv=\B^{-2}$.
\end{proof}

Note that the constant $c$ plays no role in Proposition \ref{Pr:AffineRump}.

As the following example shows, an affine rumple can admit multiple affine representations; even the underlying abelian group is not necessarily determined up to isomorphism.
\begin{example}\label{Ex:tworeps}
The rumple with multiplication table
\[
    \begin{array}{c|cccc}
         &0&1&2&3\\
        \hline
     0&   1&0&3&2\\
     1&   3&2&1&0\\
     2&   1&0&3&2\\
     3&   3&2&1&0
    \end{array}
\]
is isomorphic to both $\aff{\Z_4,2,-1,1}$ and $\aff{\Z_2^2,\left(\begin{smallmatrix}1&1\\1&1\end{smallmatrix}\right),\left(\begin{smallmatrix}1&0\\0&1\end{smallmatrix}\right),\left(\begin{smallmatrix}1\\0\end{smallmatrix}\right)}$.
\end{example}

The situation is different for affine latin rumples, however, because for any affine quasigroup, the underlying
abelian group is uniquely determined. This follows from the fact that isotopic groups are isomorphic \cite[Prop.~1.4]{Smi-book}.

When classifying affine quasigroups (and affine latin rumples in particular) up to isomorphism, the following theorem
is very useful.

\begin{theorem}[Dr\'apal {\cite[Thm. 3.2]{Dra}}]\label{Th:IsoThm}
Let $Q=\aff{G,+,\A,\B,c}$ and $Q'=\aff{G,+,\A',\B',c'}$ be affine quasigroups. Then $Q$ is isomorphic to $Q'$ if and only if there are $\alpha\in\aut{G,+}$  and $u\in\mathrm{Im}(1-\A-\B)$ such that $\A'=\A^\alpha = \alpha\A\alpha\inv$, $\B'=\B^\alpha=\alpha\B\alpha\inv$ and $c'=\alpha(c+u)$.
\end{theorem}

When $F$ is a field, the endomorphisms of the additive group $(F^n,+)$ can be identified with $n\times n$ matrices with entries in $F$, as we already did in Example \ref{Ex:tworeps}. In this context, we will denote the generic endomorphisms $\varphi$ and $\psi$ by $A$ and $B$, respectively, and \eqref{Eq:Rep} becomes $x*y = Ax+By+c$.

By Proposition \ref{Pr:AffineRump}, there is then an affine latin rumple over $(F^n,+)$ if and only if any of
the equivalent equations
\begin{align}
    [A,B]           &= A^2\,,      \label{Eq:AB} \\
    [B,A\inv]       &= I\,,        \label{Eq:BA'} \\
    [A\inv,B\inv]   &= B^{-2}   \label{Eq:A'B'}
\end{align}
have a solution in $\aut{F^n,+} = \GL_n(F)$. Here, \eqref{Eq:BA'} is a particular instance of the canonical
commutation relation used, for instance, in the matrix interpretation of the Heisenberg
uncertainty principle. It is also the defining relation of the (first) Weyl algebra \cite[p.7]{Lam}, and
so finding matrices satisfying \eqref{Eq:AffineRump1} is essentially the same as classifying modules over
the Weyl algebra with the constraint that the generators should be invertible matrices.

\begin{lemma}\label{Lm:tracezero}
Let $F$ be a field and $A$, $B\in\GL_n(F)$. If $\aff{F^n,+,A,B,c}$ is an affine latin rumple then the matrices
$A$, $A^2$, $B\inv$ and $B^{-2}$ have trace $0$.
\end{lemma}
\begin{proof}
From \eqref{Eq:AB}, we have $\tr{A^2} = \tr{AB}-\tr{BA}=0$. Also, $A = A\inv A^2 = A\inv [A,B] = B-A\inv BA$, and so $\tr{A}=\tr{B}-\tr{A\inv BA} = 0$.
From \eqref{Eq:A'B'}, we have $\tr{B^{-2}} = 0$. Finally, $B\inv = B B^{-2} = B[A\inv,B\inv] = BA\inv B\inv - A\inv$, and so $\tr{B\inv} = 0$.
\end{proof}

\begin{table}
\caption{Affine latin rumples of small orders}\label{Tb:ALR}

\[
\mathrm{Aff}\left(\mathbb Z_2^2,
\left(\begin{smallmatrix}0&1\\1&0\end{smallmatrix}\right),
\left(\begin{smallmatrix}1&0\\1&1\end{smallmatrix}\right),
\left(\begin{smallmatrix}0\\0\end{smallmatrix}\right)
\right)\quad\quad
\mathrm{Aff}\left(\mathbb Z_2^2,
\left(\begin{smallmatrix}0&1\\1&0\end{smallmatrix}\right),
\left(\begin{smallmatrix}1&0\\1&1\end{smallmatrix}\right),
\left(\begin{smallmatrix}1\\0\end{smallmatrix}\right)
\right)
\]

\bigskip

\begin{align*}
\mathrm{Aff}\left(\mathbb Z_2^4,
\left(\begin{smallmatrix}0&1&0&0\\1&0&0&0\\0&0&0&1\\0&0&1&0\end{smallmatrix}\right),
\left(\begin{smallmatrix}0&0&0&1\\1&0&1&0\\0&1&0&1\\1&0&0&0\end{smallmatrix}\right),
\left(\begin{smallmatrix}0\\0\\0\\0\end{smallmatrix}\right)
\right)\quad\quad
&
\mathrm{Aff}\left(\mathbb Z_2^4,
\left(\begin{smallmatrix}0&1&0&0\\1&0&0&0\\0&0&0&1\\0&0&1&0\end{smallmatrix}\right),
\left(\begin{smallmatrix}0&0&0&1\\1&0&1&0\\0&1&0&1\\1&0&0&0\end{smallmatrix}\right),
\left(\begin{smallmatrix}0\\0\\0\\1\end{smallmatrix}\right)
\right)\\
\mathrm{Aff}\left(\mathbb Z_2^4,
\left(\begin{smallmatrix}0&1&0&0\\1&0&0&0\\0&0&0&1\\0&0&1&0\end{smallmatrix}\right),
\left(\begin{smallmatrix}0&0&0&1\\1&0&1&0\\0&1&1&1\\1&0&0&1\end{smallmatrix}\right),
\left(\begin{smallmatrix}0\\0\\0\\0\end{smallmatrix}\right)
\right)\quad\quad
&
\mathrm{Aff}\left(\mathbb Z_2^4,
\left(\begin{smallmatrix}0&1&0&0\\1&0&0&0\\0&0&0&1\\0&0&1&0\end{smallmatrix}\right),
\left(\begin{smallmatrix}0&0&0&1\\1&0&1&0\\1&0&0&0\\0&1&1&0\end{smallmatrix}\right),
\left(\begin{smallmatrix}0\\0\\0\\0\end{smallmatrix}\right)
\right)\\
\mathrm{Aff}\left(\mathbb Z_2^4,
\left(\begin{smallmatrix}0&1&0&0\\1&0&0&0\\0&0&0&1\\0&0&1&0\end{smallmatrix}\right),
\left(\begin{smallmatrix}0&0&0&1\\1&0&1&0\\1&0&0&0\\0&1&1&0\end{smallmatrix}\right),
\left(\begin{smallmatrix}0\\0\\0\\1\end{smallmatrix}\right)
\right)\quad\quad
&
\mathrm{Aff}\left(\mathbb Z_2^4,
\left(\begin{smallmatrix}0&0&0&1\\1&0&0&0\\0&1&0&0\\0&0&1&0\end{smallmatrix}\right),
\left(\begin{smallmatrix}0&0&1&0\\1&0&0&1\\1&0&0&0\\0&1&1&0\end{smallmatrix}\right),
\left(\begin{smallmatrix}0\\0\\0\\0\end{smallmatrix}\right)
\right)\\
\mathrm{Aff}\left(\mathbb Z_2^4,
\left(\begin{smallmatrix}0&0&0&1\\1&0&0&0\\0&1&0&0\\0&0&1&0\end{smallmatrix}\right),
\left(\begin{smallmatrix}0&0&1&0\\1&0&0&1\\1&0&0&0\\0&1&1&0\end{smallmatrix}\right),
\left(\begin{smallmatrix}0\\0\\0\\1\end{smallmatrix}\right)
\right)\quad\quad
&
\mathrm{Aff}\left(\mathbb Z_2^4,
\left(\begin{smallmatrix}0&0&0&1\\1&0&0&0\\0&1&0&0\\0&0&1&0\end{smallmatrix}\right),
\left(\begin{smallmatrix}0&0&1&0\\1&0&0&1\\1&0&0&0\\0&1&1&0\end{smallmatrix}\right),
\left(\begin{smallmatrix}0\\0\\1\\0\end{smallmatrix}\right)
\right)\\
\mathrm{Aff}\left(\mathbb Z_2^4,
\left(\begin{smallmatrix}0&0&0&1\\1&0&0&0\\0&1&0&0\\0&0&1&0\end{smallmatrix}\right),
\left(\begin{smallmatrix}0&1&1&0\\1&0&1&1\\1&0&0&1\\1&1&1&0\end{smallmatrix}\right),
\left(\begin{smallmatrix}0\\0\\0\\0\end{smallmatrix}\right)
\right)\quad\quad
&
\mathrm{Aff}\left(\mathbb Z_2^4,
\left(\begin{smallmatrix}0&0&0&1\\1&0&0&0\\0&1&0&0\\0&0&1&0\end{smallmatrix}\right),
\left(\begin{smallmatrix}0&1&1&0\\1&0&1&1\\1&0&0&1\\1&1&1&0\end{smallmatrix}\right),
\left(\begin{smallmatrix}0\\0\\0\\1\end{smallmatrix}\right)
\right)\\
\mathrm{Aff}\left(\mathbb Z_2^4,
\left(\begin{smallmatrix}0&0&0&1\\1&0&0&0\\0&1&0&1\\0&0&1&0\end{smallmatrix}\right),
\left(\begin{smallmatrix}0&0&1&0\\1&0&0&1\\1&0&1&0\\0&1&1&1\end{smallmatrix}\right),
\left(\begin{smallmatrix}0\\0\\0\\0\end{smallmatrix}\right)
\right)\quad\quad
&
\mathrm{Aff}\left(\mathbb Z_2^4,
\left(\begin{smallmatrix}0&0&0&1\\1&0&0&0\\0&1&0&1\\0&0&1&0\end{smallmatrix}\right),
\left(\begin{smallmatrix}0&1&0&0\\0&0&1&0\\0&0&0&1\\1&0&1&0\end{smallmatrix}\right),
\left(\begin{smallmatrix}0\\0\\0\\0\end{smallmatrix}\right)
\right)\\
\mathrm{Aff}\left(\mathbb Z_2^4,
\left(\begin{smallmatrix}0&0&0&1\\1&0&0&0\\0&1&0&1\\0&0&1&0\end{smallmatrix}\right),
\left(\begin{smallmatrix}0&1&1&0\\0&0&1&1\\1&0&1&1\\1&1&1&1\end{smallmatrix}\right),
\left(\begin{smallmatrix}0\\0\\0\\0\end{smallmatrix}\right)
\right)\quad\quad
&
\mathrm{Aff}\left(\mathbb Z_2^4,
\left(\begin{smallmatrix}0&0&0&1\\1&0&0&0\\0&1&0&1\\0&0&1&0\end{smallmatrix}\right),
\left(\begin{smallmatrix}0&1&1&0\\0&0&1&1\\1&0&1&1\\1&1&1&1\end{smallmatrix}\right),
\left(\begin{smallmatrix}0\\0\\0\\1\end{smallmatrix}\right)
\right)
\end{align*}

\bigskip

\begin{align*}
\mathrm{Aff}\left(\mathbb Z_3^3,
\left(\begin{smallmatrix}0&0&2\\1&0&0\\0&1&0\end{smallmatrix}\right),
\left(\begin{smallmatrix}0&1&0\\2&0&1\\2&1&0\end{smallmatrix}\right),
\left(\begin{smallmatrix}0\\0\\0\end{smallmatrix}\right)
\right)\quad\quad
&
\mathrm{Aff}\left(\mathbb Z_3^3,
\left(\begin{smallmatrix}0&0&2\\1&0&0\\0&1&0\end{smallmatrix}\right),
\left(\begin{smallmatrix}0&2&0\\2&0&2\\1&1&0\end{smallmatrix}\right),
\left(\begin{smallmatrix}0\\0\\0\end{smallmatrix}\right)
\right)\\
\mathrm{Aff}\left(\mathbb Z_3^3,
\left(\begin{smallmatrix}0&0&2\\1&0&0\\0&1&0\end{smallmatrix}\right),
\left(\begin{smallmatrix}0&2&0\\2&0&2\\1&1&0\end{smallmatrix}\right),
\left(\begin{smallmatrix}0\\0\\1\end{smallmatrix}\right)
\right)\quad\quad
&
\mathrm{Aff}\left(\mathbb Z_3^3,
\left(\begin{smallmatrix}0&0&1\\1&0&0\\0&1&0\end{smallmatrix}\right),
\left(\begin{smallmatrix}0&1&0\\2&0&1\\1&1&0\end{smallmatrix}\right),
\left(\begin{smallmatrix}0\\0\\0\end{smallmatrix}\right)
\right)\\
\mathrm{Aff}\left(\mathbb Z_3^3,
\left(\begin{smallmatrix}0&0&1\\1&0&0\\0&1&0\end{smallmatrix}\right),
\left(\begin{smallmatrix}0&1&0\\2&0&1\\1&1&0\end{smallmatrix}\right),
\left(\begin{smallmatrix}0\\0\\1\end{smallmatrix}\right)
\right)\quad\quad
&
\mathrm{Aff}\left(\mathbb Z_3^3,
\left(\begin{smallmatrix}0&0&1\\1&0&0\\0&1&0\end{smallmatrix}\right),
\left(\begin{smallmatrix}0&2&0\\2&0&2\\2&1&0\end{smallmatrix}\right),
\left(\begin{smallmatrix}0\\0\\0\end{smallmatrix}\right)
\right)
\end{align*}
\end{table}

\begin{example}
A straightforward calculation in \texttt{GAP} \cite{GAP} combining Proposition \ref{Pr:AffineRump} and Theorem \ref{Th:IsoThm} allows us to determine all affine latin rumples over $\Z_2^2$, $\Z_2^4$, $\Z_3^3$ and $\Z_4^2\times\Z_2^2$ up to isomorphism. For the first three abelian groups, see Table \ref{Tb:ALR}. There are $18$ affine latin rumples over the group $\Z_4^2\times\Z_2^2$. Larger abelian groups that admit affine latin rumples are beyond the reach of standard \texttt{GAP} routines.
\end{example}

\subsection{The spectrum of affine latin rumples}
\label{subsec:spectrum}

For an abelian group $G$, let $\nalr{G}$ be the number of affine latin rumples over $G$ up to isomorphism.
For a positive integer $n$, let $\nalr{n}$ be the number of affine latin rumples of size $n$ up to isomorphism.
In this section we determine the spectrum of finite affine latin rumples, that is, the set
$\setof{n\in\mathbb{N}}{\nalr{n}>0}$. We also show that $\nalr{G}=\prod_{p}\nalr{G_p}$, where $G_p$ are
$p$-primary components of $G$; $\nalr{\Z_p^k}>0$ if and only if $p$ divides $k$; $\nalr{\Z_{p^{a_1}}^{b_1}\times\cdots\times\Z_{p^{a_r}}^{b_r}}=0$ if $p$ does not divide some $b_i$;
and $\nalr{\Z_n}=0$.

\begin{proposition}\label{Pr:Decomposition}
Let $G=\prod_p G_p$ be a decomposition of a finite abelian group $G$ into its $p$-primary components $G_p$. Then $\nalr{G}=\prod_p\nalr{G_p}$.
\end{proposition}
\begin{proof}
This is immediate from the fact that $\aut{G}$ is isomorphic to $\prod_p\aut{G_p}$.
\end{proof}

Given a permutation $\pi$ of $\{1,\dots,n\}$, the associated permutation matrix $P_{\pi}$ is defined by
\begin{displaymath}
    P_{\pi}(x_1,\dots,x_n)^T = (x_{\pi(1)},\dots,x_{\pi(n)})^T.
\end{displaymath}

\begin{proposition}\label{Pr:Char}
Let $F$ be a field. There exists a solution $A$, $B\in \GL_n(F)$ of the equation \eqref{Eq:AB} if and only if
$F$ has positive characteristic dividing $n$.
\end{proposition}
\begin{proof}
Suppose that $A$, $B\in\GL_n(F)$ satisfy \eqref{Eq:AB}. From the equivalent identity \eqref{Eq:BA'} we see that
$n=tr(I) = tr(BA\inv)-tr(A\inv B)=0$, which can happen only if $F$ has positive characteristic dividing $n$.

Conversely, suppose that $F$ has positive characteristic dividing $n$. Let $A = (a_{i,j}) = P_\pi$ for $\pi=(1,\dots,n)^{-1}$, i.e., $a_{i+1,i}=1=a_{1,n}$ for all $i=1,\dots,n-1$ and $a_{i,j}=0$ otherwise. Let $B=I-D$, where $D=(d_{i,j})$ is the matrix defined by $d_{i+1,i}=i$ for all $i=1,\dots,n-1$ and $d_{i,j}=0$ otherwise. Then
\[
    [B,A\inv]=(I-D)A\inv -A\inv (I-D)=A\inv D-D A\inv =I,
\]
where the last equality follows from the fact that $A\inv D$ is a diagonal matrix with diagonal
$(1,2,\dots,n-1,0)$, $DA\inv$ is a diagonal matrix with diagonal $(0,1,2,\dots,n-1)$,
and $0-(n-1)=1$ since the characteristic of $F$ divides $n$.
\end{proof}

The following lemma is standard \cite[Thms.~41~and~42]{KN2}. Recall that a \emph{congruence} of a quasigroup $(X,\cdot,\ldv,\rdv)$ is an equivalence relation $\sim$ on $X$ such that $xy\sim uv$, $x\ldv y\sim u\ldv v$
and $x\rdv y\sim u\rdv v$ whenever $x\sim u$ and $y\sim v$.

\begin{lemma}\label{Lm:Factor}
Let $X=\aff{G,+,\A,\B,c}$ be an affine quasigroup. The congruences of $X$ are in one-to-one correspondence with
subgroups of $(G,+,0)$ that are invariant under $\A$ and $\B$. Given a congruence $\sim$, the corresponding
subgroup is the equivalence class of $\sim$ containing $0$. Given a subgroup $H$, the corresponding congruence
is defined by $x\sim y$ if and only if $x-y\in H$.
\end{lemma}

If $X=\aff{G,+,\A,\B,c}$ is an affine latin rumple and $H$ is a subgroup of $(G,+)$ invariant under $\A$ and $\B$, we
denote by $X/H$ the factor of $X$ modulo the congruence of $X$ corresponding to $H$.

\begin{proposition}\label{Pr:Structure}
Let $p$ be a prime and let $X$ be an affine latin rumple over
$\Z_{p^{a_1}}^{b_1}\times\cdots\times\Z_{p^{a_r}}^{b_r}$. Then each $b_1$, $\dots$, $b_r$ is divisible by $p$.
\end{proposition}
\begin{proof}
Let $G=\Z_{p^{a_1}}^{b_1}\times\cdots\times\Z_{p^{a_r}}^{b_r}$ with $a_1>a_2>\cdots>a_r$. We may assume without
loss of generality that $(X,*)=\aff{G,\A,\B,0}$ for some $\A$, $\B\in\aut{G}$. The subset $pG=\setof{px}{x\in G}$ is a
characteristic subgroup of $G$ and hence invariant under $\A$ and $\B$. By Lemma \ref{Lm:Factor}, the rumple
$X/pG$ is affine over the group $G/pG\cong \Z_p^{b_1+\cdots+b_r}$. By Proposition \ref{Pr:Char}, $p$ divides
$\sum_{i=1}^r b_i$.

The map $x\mapsto px$ is an endomorphism of $X$ as $p(x*y) = p(\A x+\B y) = \A(px)+\B(py) = (px)*(py)$.
The image $pX$ is isomorphic to a quotient of $X$, and hence is affine over $pG$ by Lemma \ref{Lm:Factor}.
Applying $p$ repeatedly, we conclude that the rumple $p^{a_r}X$ is affine
over the group $p^{a_r}G$, which is isomorphic to $\Z_{p^{c_1}}^{b_1}\times\cdots\times \Z_{p^{c_{r-1}}}^{b_{r-1}}$
for suitable $c_1$, $\dots$, $c_{r-1}>0$. As above, we deduce that $p$ divides $\sum_{i=1}^{r-1}b_i$.

Hence $p$ divides $b_r$. Repeating the above argument with $p^{a_r}X$ instead of $X$ finishes the proof.
\end{proof}

\begin{proposition}\label{Pr:Spectrum}
Let $p$ be a prime. An affine latin rumple of order $p^k$ exists if and only if $p$ divides $k$.
\end{proposition}
\begin{proof}
If $p$ divides $k$ then Proposition \ref{Pr:Char} furnishes an example over the elementary abelian group $\Z_p^k$. For the converse, suppose that there is an affine latin rumple over $G=\Z_{p^{a_1}}^{b_1}\times\cdots\times\Z_{p^{a_r}}^{b_r}$ with $|G|=p^{a_1b_1+\cdots +a_rb_r}=p^k$. By Proposition \ref{Pr:Structure}, $p$ divides each of $b_1$, $\dots$, $b_r$ and thus $p$ divides $k$.
\end{proof}

Combining Propositions \ref{Pr:Decomposition} and \ref{Pr:Spectrum}, we obtain:

\begin{theorem}\label{Th:Spectrum}
Let $p_1$, $\dots$, $p_m$ be distinct primes. An affine latin rumple of order $p_1^{k_1}\cdots p_m^{k_m}$ exists if and only if $p_i$ divides $k_i$ for every $1\le i\le m$.
\end{theorem}

We do not fully understand for which finite abelian groups $G$ we get $\nalr{G}=0$. For instance, among the abelian groups of order $64$, Proposition \ref{Pr:Structure} guarantees $\nalr{G}=0$ for all groups except for $G=\Z_8^2$, $\Z_4^2\times\Z_2^2$ and $\Z_2^6$, while Proposition \ref{Pr:Char} yields $\nalr{\Z_2^6}>0$. Computer calculations show that $\nalr{\Z_8^2}=0$ and $\nalr{\Z_4^2\times\Z_2^2}=18$.

\begin{problem}
For which finite abelian groups $G$ is there a latin rumple affine over $G$?
\end{problem}

We conclude this subsection with a supplemental nonexistence result.

\begin{lemma}\label{Lm:cyclic}
There are no affine latin rumples over cyclic groups.
\end{lemma}
\begin{proof}
Since all endomorphisms of a cyclic group $G$ commute, we have $[\A,\B]=0$ for every
$\A$, $\B\in\End{G}$. If $\A^2=[\A,\B]$ holds, $\A$ cannot be invertible.
\end{proof}

\subsection{A class of affine latin rumples}


In this subsection we expand upon the example from the proof of Proposition \ref{Pr:Char}. Recall that a square matrix is a \emph{circulant} if it is constant on all broken diagonals, and denote by $\mathrm{Circ}(c_1,\dots,c_n)$ the $n\times n$ circulant matrix with first row equal to $(c_1,\dots,c_n)$. As in the proof of Proposition \ref{Pr:Char}, let $D=(d_{i,j})$ be the $n\times n$ matrix defined by $d_{i+1,i}=i$ for all $i=1,\dots,n-1$ and $d_{i,j}=0$ otherwise.

\begin{lemma}\label{Lm:Counterpart}
Let $A{=}\mathrm{Circ}(0,\dots,0,1)$ be the permutation matrix corresponding to the $n$-cycle $(1,\dots,n)\inv$. Then an $n\times n$ matrix $B$ satisfies $[A,B]=A^2$ if and only if $B=\mathrm{Circ}(c_1,\dots,c_n)-D$ for some $c_1$, $\dots$, $c_n$.
\end{lemma}
\begin{proof}
Since $A$ is invertible, we can work with the equivalent identity $[B,A\inv]=I$ instead. Using $A^{-1}=\mathrm{Circ}(0,1,0,\dots,0)$ and $B=(b_{i,j})$, we have
\[
    BA\inv - A\inv B = \left(\begin{array}{cccc}
        b_{1,n}-b_{2,1},&b_{1,1}-b_{2,2},&\cdots&b_{1,n-1}-b_{2,n}\\
        b_{2,n}-b_{3,1},&b_{2,1}-b_{3,2},&\cdots&b_{2,n-1}-b_{3,n}\\
        \vdots&\vdots&\vdots&\vdots\\
        b_{n-1,n}-b_{n,1},&b_{n-1,1}-b_{n,2},&\cdots&b_{n-1,n-1}-b_{n,n}\\
        b_{n,n}-b_{1,1},&b_{n,1}-b_{1,2},&\cdots&b_{n,n-1}-b_{1,n}
    \end{array}\right).
\]
Then $[B,A\inv]=I$ holds if and only if we have (reading off the main diagonal)
\begin{equation}\label{Eq:MainDiag}
    b_{1,n}-b_{2,1}=b_{2,1}-b_{3,2}=\cdots=b_{n-1,n-2}-b_{n,n-1}=b_{n,n-1}-b_{1,n}=1,
\end{equation}
and (reading off the broken diagonal just above the main diagonal)
\begin{equation}\label{Eq:BrokenDiag}
    b_{1,1}-b_{2,2} = b_{2,2}-b_{3,3} = \cdots = b_{n-1,n-1}-b_{n,n} = b_{n,n}-b_{1,1}=0,
\end{equation}
and similarly on the remaining broken diagonals. All solutions of the linear system \eqref{Eq:MainDiag} are of the form $b_{1,n}=c$, $b_{2,1}=c-1$, $\dots$, $b_{n,n-1}=c-n$ for some $c$, while all solutions of \eqref{Eq:BrokenDiag} are of the form $b_{1,1}=b_{2,2}=\cdots=b_{n,n}=c$ for some $c$. The claim follows.
\end{proof}

In order to construct an affine latin rumple from $A=\mathrm{Circ}(0,\dots,0,1)$ and $B=\mathrm{Circ}(c_1,\dots,c_n)-D$, we must ensure that $B$ is invertible. The following result characterizes invertible matrices of the form $\mathrm{Circ}(c_1,\dots,c_n)-D$ with entries in $\Z_p$ in the special case when $n=p$. (By Propositon \ref{Pr:Structure}, the case $n=p$ is
precisely the case we care about.)

\begin{proposition}\label{Pr:Determinant}
Let $p$ be a prime, $c_1$, $\dots$, $c_p\in\Z_p$ and $B=\mathrm{Circ}(c_1,\dots,c_p)-D$. Then $\det(B) \equiv c_1+\dots+c_{p-1} \pmod p$.
\end{proposition}
\begin{proof}
Call a selection $P$ of $p$ cells from the square $\Z_p\times\Z_p$ a permutation pattern if every row and every column contain precisely one cell from $P$. Given a permutation pattern $P$ and an integer $k$, let $P^k$ be the pattern with cells $\setof{(i+k,j+k)}{(i,j)\in P}$, where we add coordinates modulo $p$. Let $[P]=\setof{P^k}{k\in\mathbb Z}$. We will add contributions to $\det(B)$ in groups corresponding to the classes $[P]$ of permutation patterns. Observe that all permutations corresponding to the patterns in a given class $[P]$ have the same sign since they have the same cycle structure.

Suppose that $P$ is a (broken) diagonal so that $[P]=\{P\}$. If the diagonal in $B$ corresponding to $P$ is constant with all entries equal to $c_i$, for some $1\le i\le p-1$, then its contribution to $\det(B)$ is $c_i^p\equiv c_i\pmod p$. In the nonconstant case the contribution of $P$ is $c_p(c_p-1)\cdots(c_p-(p-1))\equiv 0\pmod p$ since one of the factors is equal to $0$.

Now suppose that $P$ is not a diagonal. We claim that $P=P^m$ if and only if $p$ divides $m$ and thus $[P]=\setof{P^k}{0\le k<p}$. Indeed, if $P=P^m$, $\gcd(m,p)=1$ and $(i,j)\in P$, then $P$ must contain the distinct cells $(i+km,j+km)$, $0\le k<p$, and hence $P$ is a diagonal. Suppose that $P$ intersects the nonconstant diagonal of $B$ in $d$ cells. If $d=0$ then every $P^k$ contributes the same amount to $\det(B)$ and hence the contribution of $[P]$ is congruent to $0$ modulo $p$. We can therefore assume that $d>0$ and note that $d\le p-2$ because if $P$ contains $p-1$ cells from the nonconstant diagonal of $B$ then $P$ must also contain the last cell from the nonconstant diagonal, a contradiction. The contribution of $P$ is then of the form $\pm c_{i_1}\cdots c_{i_{p-d}}(c_p-j_1)\cdots(c_p-j_d)$, where $1\le i_k<p$ and $0\le j_d<p$, while the contribution of $P^k$ is $\pm c_{i_1}\cdots c_{i_{p-d}}(c_p-j_1-k)\cdots(c_p-j_d-k)$. The combined contribution of $[P]$ is therefore $\pm c_{i_1}\cdots c_{i_{p-d}}\cdot s$, where
\[
    s=\sum_{0\le k<p} (c_p-j_1-k)\cdots (c_p-j_d-k) \equiv \sum_{0\le k<p} (c_p-j_1+k)\cdots (c_p-j_d+k).
\]
We will show that $s\equiv 0\pmod p$, finishing the proof.

For $1\le i\le d$, let $e_i=c_p-j_i$ so that $s = \sum_{0\le k<p}(e_1+k)\cdots(e_d+k)$. Let us view $s$ as a polynomial in variables $e_1$, $\dots$, $e_d$ and let us determine the coefficients of all monomials. The monomial $e_1\cdots e_d$ has coefficient $1+1+\cdots+1=p\equiv 0\pmod p$. Every monomial of the form $e_{i_1}\cdots e_{i_\ell}$ with $0\le \ell<d$ has coefficient $0+1^{d-\ell}+2^{d-\ell}+\cdots +(p-1)^{d-\ell}$.

It now suffices to show that $1^t+2^t+\cdots +(p-1)^t\equiv 0 \pmod p$ for every $1\le t\le p-2$ since we have already observed that $1\le d-\ell\le d\le p-2$. Let $\omega$ be a primitive $(p-1)$st root of unity in $\mathbb Z_p$. Then $1^t+2^t+\cdots +(p-1)^t = 1^t + \omega^t + \omega^{2t}+ \cdots +\omega^{(p-2)t} = (1-\omega^{(p-1)t})(1-\omega^t)\inv \equiv 0\pmod p$ since $\omega^{p-1}=1$ and $\omega^t\ne 1$.
\end{proof}

\begin{corollary}\label{Cr:pp}
Let $A=\mathrm{Circ}(0,\dots,0,1)$ and $B=(c_1,\dots,c_p)-D$ be $p\times p$ matrices, where $c_1$, $\dots$, $c_p\in\Z_p$ satisfy $c_1+\cdots+c_{p-1}\not\equiv 0\pmod p$. Then for every $c\in\Z_p$, $\aff{\Z_p^p,A,B,c}$ is an affine latin rumple of order $p^p$.
\end{corollary}

\begin{remark}
The isomorphism problem for affine latin rumples of the form $\aff{\Z_p^p,A,B,c}$ with $A=\mathrm{Circ}(0,\dots,0,1)$ and $B=(c_1,\dots,c_p)-D$ is tractable for small values of $p$. It is also possible to generalize the construction of Corollary \ref{Cr:pp} further by considering matrices that do not differ much from $\mathrm{Circ}(0,\dots,0,1)$, say $A=\mathrm{Circ}(0,\dots,0,1)+aE_{i+1,i}$, where $E_{i,j}$ is the matrix whose only nonzero entry $1$ is located in row $i$ and column $j$. One can then obtain statements analogous to Lemma \ref{Lm:Counterpart} and Proposition \ref{Pr:Determinant}. The details will be presented elsewhere.
\end{remark}

\subsection{A characterization of affine latin rumples}

In this subsection we obtain a characterization of affine latin rumples among latin rumples in terms of the displacement group and the multiplication group. According to Proposition \ref{Pr:dis}, for every rumple $X$, the displacement group $\dis{X}$ is normal in $\lmlt{X}$, and thus $\dis{X}$ is normal in $\mlt{X}$ if and only if $\dis X^{R_x^{\pm1}}\subseteq\dis X$ for every $x\in X$.

Recall that a permutation group $G$ acts \emph{regularly} on $X$ if for every $x$, $y\in X$ there is a unique $g\in G$ such that $g(x)=y$.

\begin{theorem} \label{Th:affine_char}
The following conditions are equivalent for a latin rumple $X$:
\begin{enumerate}
\item $X$ is affine;
\item $\dis X$ is abelian and normal in $\mlt X$.
\end{enumerate}
\end{theorem}
\begin{proof}
Suppose that (1) holds. In an affine rumple $(X,*)=\aff{G,\A,\B,c}$, we have
\[
    L_xL_y\inv(z) = \A(x)+\B(\B\inv(z-\A(y)-c))+c = \A(x)-\A(y)+z.
\]
Hence, since $\A$ is surjective, we have $\dis{X} = \setof{\alpha_x}{x\in X}$, where $\alpha_x(z)=x+z$. It is now clear that $\dis{X}$ is an abelian group. Moreover,
\begin{align*}
    \alpha_x^{R_y}(z)       &= \alpha_x(z/y)*y=\A(x+\A\inv(z-\B(y)-c))+\B(y)+c=\A(x)+z, \\
    \alpha_x^{R_y\inv}(z)  &= \alpha_x(z*y)/y=\A\inv((\A(z)+\B(y)+c)+x-\B(y)-c)=\A\inv(x)+z
\end{align*}
shows that $\alpha_x^{R_y} = \alpha_{\A(x)}$ and $\alpha_x^{R_y\inv} = \alpha_{\A\inv(x)}$ are elements of $\dis{X}$.

Now suppose that (2) holds. Pick $e\in X$ arbitrarily and let $G=\dis X$, $\A(f)=f^{R_{ee}}$, $\B(f)=f^\sigma$, where $\sigma=R_{ee}L_eR_e\inv$ (cf. Proposition \ref{Pr:latin_2-div}), and $c=L_{ee}L_e\inv$. We will show that $X$ is isomorphic to $\aff{G,\A,\B,c}$. First observe that both $\A$, $\B$ are well-defined because $\dis{X}\unlhd\mlt{X}$.
Consider the map
\[
    \xi:X\to\aff{G,\A,\B,c},\qquad x\mapsto L_xL_e\inv
\]
and note that $\xi$ is injective since $X$ is a (left) quasigroup. The identity $L_{y\rdv(e\ldv x)}L_e\inv(x) = y$ shows that
$\dis{X}$ is transitive and hence regular, being abelian. This implies that $G=\dis{X}=\setof{L_xL_e\inv}{x\in G}$ and thus
that $\xi$ is bijective. It remains to prove that $\xi$ is a homomorphism. We want to show that $\xi(x)*\xi(y)=\A(L_xL_e\inv)\B(L_yL_e\inv)c=(L_xL_e\inv)^{R_{ee}}(L_yL_e\inv)^{\sigma}(L_{ee}L_e\inv)$ is equal to $\xi(xy)=L_{xy}L_e\inv$. Since $\dis X$ is regular, it is sufficient to check that the two permutations agree at a single point, for instance at $e\cdot ee$. Now,
\begin{align*}
    (\xi(x)*\xi(y))(e\cdot ee)
    &=(L_xL_e\inv)^{R_{ee}}(L_yL_e\inv)^{\sigma}(L_{ee}L_e\inv)(e\cdot ee) = (L_xL_e\inv)^{R_{ee}}\sigma L_yL_e\inv\sigma\inv(ee\cdot ee)  \\
    &= (L_xL_e\inv)^{R_{ee}}\sigma L_yL_e\inv(ee) = (L_xL_e\inv)^{R_{ee}}\sigma(ye) \\
    & = (L_xL_e\inv)^{R_{ee}}(ye\cdot ye) \overset{\eqref{Eq:LeftRump}}{=} (L_xL_e\inv)^{R_{ee}}(ey\cdot ee)\\
    & = R_{ee}L_xL_e\inv(ey) = xy\cdot ee = L_{xy}L_e\inv(e\cdot ee) = \xi(xy)(e\cdot ee).\qedhere
\end{align*}
\end{proof}

\begin{corollary}
The following conditions are equivalent for a latin rumple $X$:
\begin{enumerate}
\item $X$ is linear;
\item $X$ contains an idempotent element and $\dis X$ is abelian and normal in $\mlt X$.
\end{enumerate}
\end{corollary}

\begin{proof}
If $X$ is linear then $0$ is an idempotent element. Conversely, in the construction in the proof of Theorem \ref{Th:affine_char}, use an idempotent element $e$ and observe that $c=L_{ee}L_{e}\inv=1$, the identity element in
$\dis{X}$.
\end{proof}


\section{Latin rumples isotopic to groups}\label{sec:rlinear}

For the purposes of the present section, we extend the definition of linear representation. Let $(G,\circ)$ be an arbitrary loop, not necessarily associative or commutative. A binary algebra $(G,*)$ is called \emph{right linear} (resp. \emph{left linear}) over $(G,\circ)$ if there exist $\varphi:G\to G$ and $\psi\in\mathrm{End}(G,\circ)$ (resp. $\varphi\in\mathrm{End}(G,\circ)$ and $\psi:G\to G$) such that \[ x*y=\varphi(x)\circ\psi(y) \]
for all $x$, $y\in G$. As in the case of linear representations, note that $(G,*)$ is a left quasigroup if and only if $\varphi$ is bijective, and a right quasigroup if and only if $\psi$ is bijective.

Let us recall basic facts about loop isotopes (see \cite{Smi-book} for details). For a quasigroup $X$, fix $e,f\in X$ and define a binary operation $\circ_{e,f}:X\times X\to X$ by \[ x\circ_{e,f} y = (x\rdv e)(f\ldv y).\]
Then $(X,\circ_{e,f})$ is a loop with identity element $fe$ and $X$ is isotopic to $(X,\circ_{e,f})$.
Loop isotopes of this form are said to be \emph{principal}. Thus every quasigroup $X$ is isotopic to a loop and every loop isotope of $X$ is
isomorphic to a principal loop isotope of $X$. Moreover, a group is isomorphic to all of its loop isotopes.
Therefore, if a quasigroup $X$ is isotopic to a group $G$ then all loop isotopes of $X$ are isomorphic to $G$.

The left multiplication group $\lmlt{X,\circ_{e,f}}$ is generated by all permutations of the form $L_{x\rdv e}L_f\inv$, $x\in X$. From this observation, we see immediately that $\disp{X} = \lmlt{X,\circ_{e,f}}$ for every $e,f\in X$.

\begin{proposition}\label{Pr:isotope}
A quasigroup $X$ is isotopic to a group if and only if $\disp{X}$ acts regularly on $X$. In such a case, $\disp{X}$ is isomorphic to all group isotopes of $X$.
\end{proposition}
\begin{proof}
Note that if $X$ is a loop with identity element $1$ then $\disp{X}=\lmlt{X}$ since $L_xL_1\inv = L_x$. Let us first show that a loop $X$ is a group if and only if $\lmlt{X}$ acts regularly on $X$. The direct implication is obvious. Conversely, suppose that $\lmlt{X}$ acts regularly on $X$. For $g\in\lmlt{X}$ there is $x\in X$ such that $g(1)=x$ and thus $g=L_x$ by regularity. Hence the composition of any two left translations $L_xL_y$ is a left translation, necessarily $L_{xy}$ on account of $L_{xy}(1)=L_xL_y(1)$. This means that $X$ is a group.

Now let $X$ be a quasigroup. If $X$ is isotopic to a group $G$ then there is a principal loop isotope $(X,\circ_{e,f})$ isomorphic to $G$. Then $\disp{X}=\lmlt{X,{\circ}_{e,f}}$ acts regularly on $X$ by the first paragraph. Conversely, suppose that $\disp{X}$ acts regularly on $X$ and let $(X,\circ_{e,f})$ be any loop isotope of $X$. Then $\lmlt{X,\circ_{e,f}}=\disp{X}$ acts regularly and hence $(X,\circ_{e,f})$ is a group by the first paragraph.
\end{proof}

Since transitive abelian permutation groups act regularly, we have the following corollary which can be traced to
Belousov \cite{Bel66}. 

\begin{corollary} \label{Cr:isotope_abelian_group}
A quasigroup $X$ is isotopic to an abelian group if and only if $\disp X$ is abelian.
\end{corollary}

\begin{theorem}\label{Th:right linear}
For a latin rumple $X$, the following conditions are equivalent:
  \begin{enumerate}
    \item $X$ is right linear over a group;
    \item $X$ is isotopic to a group;
    \item $\dis{X}$ acts regularly on $X$.
  \end{enumerate}
\end{theorem}
\begin{proof}
The equivalence of (2) and (3) follows from Proposition \ref{Pr:isotope}. Obviously, (1) implies (2).
We prove that (2) implies (1). Let $X$ be (principally) isotopic to a group $(X,\circ)$, i.e., there are permutations $\varphi,\psi$ of $X$  such that $xy = \varphi(x)\circ \psi(y)$ for all $x,y\in X$.
We may assume without loss of generality that $\psi(1) = 1$, otherwise, set $\bar{\varphi}(x) =  \varphi(x)\circ \psi(1)$ and $\bar{\psi}(y) = \psi(1)\inv\circ \psi(y)$ so that $xy = \bar{\varphi}(x)\circ \bar{\psi}(y)$. Writing \eqref{Eq:LeftRump} in terms of $\circ$, $\varphi$ and $\psi$   (and replacing $z$ with $\psi\inv(z)$),
  we have
  \[
    \varphi(\varphi(x)\circ \psi(y))\circ \psi(\varphi(x)\circ z) =
    \varphi(\varphi(y)\circ \psi(x))\circ \psi(\varphi(y)\circ z)
  \]
  for all $x,y,z\in X$. Rearranging this, we have
  \[  \varphi(\varphi(y)\circ \psi(x))\inv\circ \varphi(\varphi(x)\circ \psi(y)) =
    \psi(\varphi(y)\circ z)\circ \psi(\varphi(x)\circ z)\inv, \]
and we note that the left hand side is independent of $z$. Substituting first $z=1$ and then $z=\varphi(x)\inv$ therefore yields
  \[  \psi(\varphi(y))\circ \psi(\varphi(x))\inv = \psi(\varphi(y)\circ \varphi(x)\inv)  \]
  for all $x,y\in X$. Since $\varphi$ is a bijection, it follows that $\psi$ is an automorphism and $X$ is right linear over $(X,\circ)$.
\end{proof}

\begin{corollary}
  For a latin rumple $X$, the following conditions are equivalent:
  \begin{enumerate}
    \item $X$ is right linear over an abelian group;
    \item $X$ is isotopic to an abelian group;
    \item $\dis{X}$ is abelian.
  \end{enumerate}
\end{corollary}

We conclude this section with another characterization of affine latin rumples.

\begin{lemma}\label{Lm:BothLinear}
Let $(G,\circ)$ be a group. If a quasigroup $(G,*)$ is both left linear and right linear over $(G,\circ)$, then there are $\varphi,\psi\in\aut{G,\circ}$ and $c\in G$ such that $x*y = \varphi(x)\circ c\circ\psi(y)$ for all $x,y\in G$.
\end{lemma}
\begin{proof}
The direct implication is obvious. For the converse, suppose that for every $x$, $y\in G$ we have
\begin{equation}\label{Eq:AuxHere}
    x*y = \varphi_1(x)\circ g_1(y) = f_2(x)\circ \psi_2(y)
\end{equation}
for some bijections $g_1$, $f_2$ of $G$ and some automorphisms $\varphi_1$, $\psi_2$ of $(G,\circ)$. With $x=y=1$, \eqref{Eq:AuxHere} yields $g_1(1)=f_2(1)$ and we will call this element $c$. Define bijections $\psi_1$, $\varphi_2$ by $g_1(x)=c\circ\psi_1(x)$ and $f_2(x)=\varphi_2(x)\circ c$. Note that $\psi_1(1)=1=\varphi_2(1)$. Then \eqref{Eq:AuxHere} implies
\begin{displaymath}
    x*y = \varphi_1(x)\circ c\circ \psi_1(y) = \varphi_2(x)\circ c\circ \psi_2(y)\,.
\end{displaymath}
With $x=1$ we obtain $c\circ\psi_1(y) = c\circ\psi_2(y)$ and hence $\psi_1=\psi_2$. The equality $\varphi_1=\varphi_2$ follows by setting $y=1$. We finish the proof by taking $\varphi=\varphi_1$ and $\psi=\psi_2$.
\end{proof}

\begin{theorem}\label{Th:left linear}
A latin rumple is affine if and only if it is left linear over a group.
\end{theorem}
\begin{proof}
Let $(X,*)$ be a latin rumple. The necessity is obvious, so assume $(X,*)$ is left linear over a group $(X,\circ)$. Since $(X,*)$ is isotopic to $(X,\circ)$, it follows from Theorem \ref{Th:right linear} that
$(X,*)$ is also right linear over $(X,\circ)$. By Lemma \ref{Lm:BothLinear}, there are $\varphi,\psi\in \aut{X,\circ}$ and $c\in X$ such that
\begin{equation}\label{Eq:BothLinear}
    x*y = \varphi(x)\circ c\circ \psi(y)
\end{equation}
for all $x,y\in X$. It remains to show that $(X,\circ)$ is abelian.

Writing \eqref{Eq:LeftRump} in terms of \eqref{Eq:BothLinear} and canceling $\psi(c)\circ \psi^2(z)$ on the right, we get
\begin{equation}\label{Eq:leftlin1}
    \varphi^2(x)\circ \varphi(c)\circ \varphi\psi(y)\circ c\circ \psi\varphi(x) =
    \varphi^2(y)\circ \varphi(c)\circ \varphi\psi(x)\circ c\circ \psi\varphi(y)
\end{equation}
for all $x,y\in X$. Setting $x=1$ and rearranging yields
\begin{equation}\label{Eq:leftlin2}
    \varphi\psi(y) = [\varphi(c\inv\circ \varphi(y)\circ c)]\circ [c\circ \psi\varphi(y)\circ c\inv]
\end{equation}
for all $y\in X$. Observe that $\alpha(y) = \varphi(c\inv\circ \varphi(y)\circ c)$ and
$\beta(y) = c\circ \psi\varphi(y)\circ c\inv$ define two automorphisms of $(X,\circ)$. Now, for all $x,y\in X$,
\begin{align*}
\alpha(x)\circ \alpha(y)\circ \beta(x)\circ \beta(y)
&= \alpha(x\circ y)\circ \beta(x\circ y) = \varphi\psi(x\circ y) = \varphi\psi(x)\circ \varphi\psi(y) \\
&= \alpha(x)\circ \beta(x)\circ \alpha(y)\circ \beta(y)\,,
\end{align*}
using \eqref{Eq:leftlin2} in the second and fourth equalities. Canceling, we have
$\alpha(y)\circ \beta(x) = \beta(x)\circ \alpha(y)$ for all $x,y\in X$. Since $\alpha$ and $\beta$ are
permutations, $(X,\circ)$ is abelian.
\end{proof}

\section{Nilpotent latin rumples}
\label{sec:nilpotent}

\subsection{Central extensions}
\label{subsec:central}

All the latin rumples $X$ we have seen so far are affine, hence isotopic to an abelian group $G$
and such that $\dis{X}=G$ is abelian and normal in $\mlt{X}$. In this section we will present
a construction based on central extensions that produces examples of nonaffine latin rumples,
even latin rumples not isotopic to groups.

We adapt the general construction of central extensions from the commutator theory of universal
algebra \cite[{\S}7]{FM} to the class of rumples. (See \cite{Ven} for other types of rumple extensions.)

Let $(G,+)$ be an abelian group, $(F,\cdot)$ a left quasigroup, $\A\in\End{G,+}$, $\B\in\aut{G,+}$ and
$\theta:F\times F\to G$. A \emph{central extension} $\Ext{G,F,\A,\B,\theta}$ of $(G,+)$ by $(F,\cdot)$ is the binary algebra
$(G\times F,\ast)$ with multiplication
\[
    (a,x)\ast (b,y)=(\varphi(a)+\psi(b)+\theta(x,y),xy).
\]
Note that we recover affine rumples as a special case of central extensions by setting $F=1$.

It is easy to see that $\Ext{G,F,\A,\B,\theta}$ is a left quasigroup with
\[
    (a,x)\ldv(b,y)=(\B\inv(c-\A(a)-\theta(x,x\ldv y)),x\ldv y)
\]
and that it is latin if and only if $F$ is latin and $\A\in\aut{G,+}$. Straightforward calculation yields:

\begin{proposition}\label{Pr:CE}
Let $(G,+)$ be an abelian group, $F$ a Rump left quasigroup, $\A\in\End{G,+}$, $\B\in\aut{G,+}$ and
$\theta:F\times F\to G$. Then $\Ext{G,F,\A,\B,\theta}$ is a Rump left quasigroup if and only if
$[\A,\B]=\A^2$ and
\begin{equation}\label{Eq:Cocycle}
    \A(\theta(x,y)-\theta(y,x)) + \psi(\theta(x,z)-\theta(y,z)) + \theta(xy,xz)-\theta(yx,yz) = 0
\end{equation}
for every $x$, $y$, $z\in F$.
\end{proposition}

A Rump left quasigroup is said to be \emph{nilpotent} if it is obtained from the trivial quasigroup
by finitely many iterations of central extensions. If a nilpotent Rump left quasigroup can be obtained
in $n$ but no fewer steps, we say that it has \emph{nilpotence class} $n$. (This is in accordance with
the abstract definition of nilpotence thanks to \cite[Proposition 7.1]{FM}.)

\begin{proposition}\label{Pr:SpectrumNilpotent}
Every finite nilpotent latin rumple has order $p_1^{p_1k_1}\cdots p_r^{p_rk_r}$ for some distinct primes
$p_1,\ldots,p_r$ and integers $k_1,\ldots,k_r$.
\end{proposition}
\begin{proof}
Let $X=\Ext{G,F,\A,\B,\theta}$ be a finite nilpotent latin rumple, where we can assume that
$G$ is a nontrivial group and $F$ is a rumple of nilpotence class less than $n$, the nilpotence class
of $X$. Since $X$ is latin, $F$ is also latin and $\A$, $\B\in\aut{G}$ satisfy $[\A,\B]=\A^2$. Then for any
$c\in G$ the affine rumple $Y=\aff{G,\A,\B,c}$ is latin and hence of order
$|G|=|Y|=p_1^{p_1k_1}\cdots p_r^{p_rk_r}$ by Theorem \ref{Th:Spectrum}. If $n=1$ then $X=Y$ and we are done.
Otherwise $|F|$ and thus also $|X|=|G|\cdot |F|$ have the desired form by induction.
\end{proof}

\subsection{A class of central extensions over the Klein group}
\label{subsec:Klein}

Throughout this subsection, let $G=\Z_2\times\Z_2$ and $A$, $B\in\aut{G}$ be given by
\[
    A = \binom{0\ 1}{1\ 0},\quad B = \binom{1\ 0}{1\ 1}.
\]
We have already observed that $[A,B]=A^2$ holds.

Let $F$ be a rumple. Then a mapping $\theta:F\times F\to G$ can be written as
\[
    \theta(x,y) = \binom{\alpha(x,y)}{\beta(x,y)}
\]
for some $\alpha$, $\beta:F\times F\to \Z_2$. The cocycle condition \eqref{Eq:Cocycle} becomes
\begin{equation}\label{Eq:MatrixEquation}
    \binom{0\ 1}{1\ 0}\binom{\alpha(x,y)-\alpha(y,x)}{\beta(x,y)-\beta(y,x)}
    + \binom{1\ 0}{1\ 1}\binom{\alpha(x,z)-\alpha(y,z)}{\beta(x,z)-\beta(y,z)}
    + \binom{\alpha(xy,xz)-\alpha(yx,yz)}{\beta(xy,xz)-\beta(yx,yz)}
    = \binom{0}{0},
\end{equation}
which is equivalent to the system of linear equations
\begin{align*}
    \beta(x,y)-\beta(y,x)+\alpha(x,z)-\alpha(y,z)+\alpha(xy,xz)-\alpha(yx,yz)&=0,\\
    \alpha(x,y)-\alpha(y,x)+\alpha(x,z)-\alpha(y,z)+\beta(x,z)-\beta(y,z)+\beta(xy,xz)-\beta(yx,yz)&=0.
\end{align*}
A solution is obtained by setting
\begin{equation}\label{Eq:SpecialAB}
    \alpha=0\qquad\text{and}\qquad\beta(x,y)=\left\{\begin{array}{lr} 1,\text{ if $x=y$},\\ 0, \text{ otherwise}.\end{array}\right.
\end{equation}
Indeed, $x=z$ if and only if $yx=yz$, so $\beta(x,z)=\beta(yx,yz)$, $\beta(y,z) = \beta(xy,xz)$, and, of course, $\beta(x,y)=\beta(y,x)$.

\begin{lemma}\label{Lm:Klein}
Suppose that $G=\Z_2\times\Z_2$, $F$ is a nontrivial affine latin rumple, $A$, $B\in\aut{G,+}$ and $\theta:F\times F\to G$ are given by
\[
     A = \binom{0\ 1}{1\ 0},\quad B = \binom{1\ 0}{1\ 1},\quad
     \theta(x,y) = \binom{0}{\beta(x,y)},\quad
     \beta(x,y)=\left\{\begin{array}{lr} 1,\text{ if $x=y$},\\ 0, \text{ otherwise}.\end{array}\right.
\]
Then $\Ext{G,F,A,B,\theta}$ is a latin rumple with nonabelian $\dis{X}$. In particular, $X$ is not affine.
\end{lemma}
\begin{proof}
We have already verified that $[A,B]=A^2$ and \eqref{Eq:Cocycle} holds, so $X$ is a latin rumple. Denote a typical element of $G\times F$ by
\[
    \mathbf{x} = \left(\binom{x_1}{x_2},x\right).
\]
Straightforward calculation then yields
\begin{align*}
    L_{\mathbf x}(\mathbf y) &= \left(\binom{x_2+y_1}{x_1+y_1+y_2+\beta(x,y)},xy\right),\\
    L\inv_{\mathbf x}(\mathbf y) &= \left(\binom{-x_2+y_1}{-x_1+x_2-y_1+y_2-\beta(x,x\ldv y)},x\ldv y\right),\\
    L\inv_{\mathbf x}L_{\mathbf y}(\mathbf z) &= \left(\binom{-x_2+y_2+z_1}{-x_1+x_2+y_1-y_2+z_2+\beta(y,z)-\beta(x,x\ldv(yz))},x\ldv(yz)\right),\\
    L\inv_{\mathbf x}L_{\mathbf y}L\inv_{\mathbf u}L_{\mathbf v}(\mathbf z) &=
        \left(\binom{w_1}{w_2}, x\ldv (y(u\ldv (vz)))\right),
\end{align*}
where
\begin{align*}
    w_1&=-x_2+y_2-u_2+v_2+z_1,\\
    w_2&=-x_1+x_2+y_1-y_2-u_1+u_2+v_1-v_2+z_2\\
        &\quad+\beta(v,z)-\beta(u,u\ldv(vz)) +\beta(y,u\ldv(vz)) - \beta(x,x\ldv(y(u\ldv(vz)))).
\end{align*}
Since $F$ is affine, the group $\dis{F}$ is abelian and
\[
    x\ldv (y(u\ldv (vz))) = L_x\inv L_yL_u\inv L_v(z) = L_u\inv L_vL_x\inv L_y(z) = u\ldv (v(x\ldv (yz)))
\]
holds. We then see that $L\inv_{\mathbf x}L_{\mathbf y}L\inv_{\mathbf u}L_{\mathbf v}(\mathbf z)$ is equal to $L\inv_{\mathbf u}L_{\mathbf v}L\inv_{\mathbf x}L_{\mathbf y}(\mathbf z)$ if and only if
\begin{multline}\label{Eq:ToCommute}
    \beta(v,z)-\beta(u,u\ldv (vz)) + \beta(y,u\ldv (vz)) - \beta(x,x\ldv(y(u\ldv(vz))))\\
    = \beta(y,z)-\beta(x,x\ldv(yz)) + \beta(v,x\ldv(yz)) - \beta(u,u\ldv(v(x\ldv(yz))))
\end{multline}
for every $x$, $y$, $u$, $v$, $z\in F$. Thus $\dis{X}$ is nonabelian if \eqref{Eq:ToCommute} fails for a choice of elements of $F$.

Setting $u=y$ in \eqref{Eq:ToCommute} yields
\[
    \beta(v,z) -\beta(x,x\ldv(vz)) = \beta(y,z) - \beta(x,x\ldv(yz)) + \beta(v,x\ldv(yz)) - \beta(y,y\ldv(v(x\ldv(yz)))).
\]
Substituting $z=y\ldv x^2$ (which is equivalent to $x\ldv(yz) = x$) and using $\beta(x,x)=1$ then yields
\begin{equation}\label{Eq:FinalCommute}
    \beta(v,y\ldv x^2) - \beta(x,x\ldv(v(y\ldv x^2))) = \beta(y,y\ldv x^2)-1 + \beta(v,x)-\beta(y,y\ldv(vx)).
\end{equation}
Select $x\ne y$ in $F$ arbitrarily. Then $x^2\ne y^2$ by unique $2$-divisibility and hence $\beta(y,y\ldv x^2) = \beta(yy,y\ldv (yx^2)) = \beta(y^2,x^2)=0$. Select $v\in F$ such that $v\ne x$ (which yields $\beta(v,x)=0$), $v\ne y$ (which implies $v\ldv x^2\ne y\ldv x^2$, $x^2\ne v(y\ldv x^2)$ and $\beta(x,x\ldv(v(y\ldv x^2)))=0$), $v\ne y^2/x$ (which implies $\beta(y,y\ldv(vx))=0$) and $v\ne y\ldv x^2$ (which yields $\beta(v,y\ldv x^2)=0$). Altogether, \eqref{Eq:FinalCommute} becomes $0=-1$. When $|F|\ge 5$, it is certainly possible to select $x$, $y$ and $v\in F$ as above. When $|F|<5$ then $F=X_{4,1}$ or $F=X_{4,2}$ as in Example \ref{ex:4}. In $X_{4,1}$, choose $x=0$, $y=1$ and $v=2$. In $X_{4,2}$, choose $x=0$, $y=1$ and $v=3$.

We have proved that $\dis{X}$ is nonabelian. By Theorem \ref{Th:affine_char}, $X$ is not affine.
\end{proof}

\begin{example}\label{Ex:Klein}
Note that Lemma \ref{Lm:Klein} is only one of many possible solutions to the matrix equation \eqref{Eq:MatrixEquation}. The corresponding system of linear equations over $\Z_2$ can be solved by standard methods of linear algebra. All latin rumples $X$ below were obtained as central extensions of $\Z_2\times\Z_2$:
\begin{itemize}
\item $X$ of order $16$ with $\dis{X}=\mathbb Z_2\times Q_8$, where $Q_8$ is the quaternion group,
\item $X$ of order $16$ with $\dis{X}$ abelian but not normal in $\mlt{X}$,
\item $X$ of order $64$ not isotopic to a group and satisfying the right Rump identity \eqref{Eq:RightRump},
\item $X$ of order $108$ with $\dis{X}$ a nonnilpotent group.
\end{itemize}
\end{example}

\section{Both-sided rumples}\label{sec:both-sided}

Recall that the two four-element latin rumples of Example \ref{ex:4} satisfy the right Rump identity \eqref{Eq:RightRump}.
In this section we investigate in a systematic way left quasigroups satisfying both \eqref{Eq:LeftRump} and \eqref{Eq:RightRump}. Our first result will show that such
left quasigroups are automatically latin rumples.

\subsection{The two Rump identities and the squaring map}

\begin{lemma}\label{Lm:bs_quasi}
  Let $X$ be a left quasigroup and assume that the identity
  \begin{equation}\label{Eq:bs_quasi} 
    ((x\ldv y)\ldv y)x = y
  \end{equation}
  holds for all $x,y\in X$. Then $X$ is a quasigroup.
\end{lemma}
\begin{proof}
Define an operation $\rdv$ on $X$ by setting
\begin{equation}\label{Eq:bs_rdv}
y\rdv x = (x\ldv y)\ldv y
\end{equation}
and note that \eqref{Eq:bs_quasi} immediately implies $(y\rdv x)x = y$. Dividing by $(x\ldv y)\ldv y$ on the left
in \eqref{Eq:bs_quasi} yields $((x\ldv y)\ldv y)\ldv y = x$, and thus $yx\rdv x = (x\ldv yx)\ldv yx = ((y\ldv yx)\ldv yx)\ldv yx = y$.
\end{proof}

\begin{proposition}\label{Pr:BSR}
The following conditions are equivalent for a left quasigroup $X$:
\begin{enumerate}
\item $X$ satisfies \eqref{Eq:LeftRump} and \eqref{Eq:RightRump};
\item $X$ is a rumple satisfying \eqref{Eq:RightRump};
\item $X$ is a latin rumple satisfying \eqref{Eq:RightRump}.
\end{enumerate}
In these equivalent situations, the right division operation is given by \eqref{Eq:bs_rdv}.
\end{proposition}
\begin{proof}
Obviously, (3) $\Rightarrow$ (2) $\Rightarrow$ (1). Suppose that (1) holds and let us establish (3) by showing that
\eqref{Eq:bs_quasi} holds. Indeed, we have
\begin{align*}
    (x\ldv y)x\cdot ((x\ldv y)\ldv y)x
    &\overset{\eqref{Eq:RightRump}}{=}
    (x\ldv y)((x\ldv y)\ldv y)\cdot x((x\ldv y)\ldv y)
    = y \cdot x((x\ldv y)\ldv y) \\
    &
    \overset{\phantom{\eqref{Eq:RightRump}}}{=}
    x(x\ldv y)\cdot x((x\ldv y)\ldv y)
    \overset{\eqref{Eq:LeftRump}}{=}
    (x\ldv y)x\cdot (x\ldv y)((x\ldv y)\ldv y)
    = (x\ldv y)x\cdot y,
\end{align*}
from which \eqref{Eq:bs_quasi} follows upon canceling $(x\ldv y)x$ on the left. By Lemma \ref{Lm:bs_quasi},
$X$ is a quasigroup. By Proposition \ref{Pr:latin_2-div}, $X$ is a rumple.
\end{proof}

A \emph{both-sided rumple} is a left quasigroup satisfying any of the three equivalent conditions of Proposition \ref{Pr:BSR}.

It follows from Proposition \ref{Pr:BSR} that the notion of both-sided rumple is self-dual. That is, if $(X,\cdot)$
is a both-sided rumple, then so is $(X,\cdot_{\mathrm{op}})$ with $x\cdot_{\mathrm{op}} y = y\cdot x$. Thus if an identity
holds in a both-sided rumple then its mirror image also holds. We will occasionally appeal to this observation.


\begin{proposition}\label{Pr:BSRsigma}
Let $X$ be a both-sided rumple and let $\sigma$ be the squaring map on $X$. Then:
\begin{enumerate}
\item $\sigma$ is an antiautomorphism of $X$.
\item $\sigma^2$ is an automorphism of $X$.
\item $\sigma^2(x) = xx\cdot xx = yy\cdot yx = xy\cdot yy = yx\cdot xy$ for every $x$, $y\in X$.
\item $\sigma^2 = L_{yy}L_y = R_{yy}R_y$ for every $y\in X$.
\end{enumerate}
\end{proposition}

\begin{proof}
Note that (2) follows from (1), and (4) follows from (3). Let us prove (1). For every $x$, $y\in X$ we have
\begin{align*}
    (xy\cdot yy)(yx\cdot yx)
    &\overset{\eqref{Eq:LeftRump}}{=} (xy\cdot yy)(xy\cdot xx)
    \overset{\eqref{Eq:LeftRump}}{=} (yy\cdot xy)(yy\cdot xx) \\
    &\overset{\eqref{Eq:RightRump}}{=} (yx\cdot yx)(yy\cdot xx)
    \overset{\eqref{Eq:LeftRump}}{=} (xy\cdot xx)(yy\cdot xx)
    \overset{\eqref{Eq:RightRump}}{=} (xy\cdot yy)(xx\cdot yy)
\end{align*}
and we deduce $\sigma(yx) = \sigma(x)\sigma(y)$ upon canceling $xx\cdot yy$ on the left. For (3), we compute
\[
    (y\ldv xx)y\cdot (y\ldv xx)y
    \overset{\eqref{Eq:LeftRump}}{=} y(y\ldv xx)\cdot yy
    = xx\cdot yy
    \overset{(1)}{=} yx\cdot yx.
\]
Taking square roots of both sides, we obtain $(y\ldv xx)y = yx$ and therefore
\[
    yy\cdot yx
    = yy\cdot (y\ldv xx)y
    \overset{\eqref{Eq:RightRump}}{=} y(y\ldv xx)\cdot y(y\ldv xx)
    = xx\cdot xx\,.
\]
A dual argument yields $xx\cdot xx = xy\cdot yy$. Finally, substituting $xy$ for $x$ and $yx$ for $y$ into the established identity $yx = (y\ldv xx)y$ yields
\[
    yx\cdot xy
    = (yx\ldv (xy\cdot xy))\cdot yx
    \overset{\eqref{Eq:LeftRump}}{=}
    (yx\ldv (yx\cdot yy))\cdot yx
    = yy\cdot yx.\qedhere
\]
\end{proof}

\subsection{Both-sided rumples isotopic to groups}

\begin{lemma}\label{Lm:BSRitp}
Let $X$ be a both-sided rumple. Then $x\ldv y\cdot y/x = y$ and $x/y\cdot y\ldv x = yy$ for every $x$, $y\in X$.
\end{lemma}
\begin{proof}
The first identity follows form the right division formula in Proposition \ref{Pr:BSR}. By Proposition \ref{Pr:BSRsigma}.2, $\sigma^2$ is an automorphism with respect to multiplication and hence also with respect to left division in $X$. By Proposition \ref{Pr:BSRsigma}.3, $\sigma^2(x) = x(yy)\cdot (yy\cdot yy) = x(yy)\cdot \sigma^2(y)$, so $x(yy) = \sigma^2(x)/\sigma^2(y) = \sigma^2(x/y)$. By Proposition \ref{Pr:BSRsigma}.3 again, $\sigma^2(u) = vu\cdot uv$, from which we obtain $\sigma^2(x/y) = (y\ldv x\cdot x/y)(x/y\cdot y\ldv x)$ upon substituting $x/y$ for $u$ and $y\ldv x$ for $v$. Combining, we have
\[
    x(yy) = \sigma^2(x/y) = (y\ldv x\cdot x/y)(x/y\cdot y\ldv x) = x(x/y\cdot y\ldv x)
\]
and we obtain the second identity from the statement by canceling $x$ on the left.
\end{proof}

\begin{corollary}\label{Cr:BSRitp}
Let $X$ be a both-sided rumple. Then for each $e\in X$, the principal loop isotope $(X,\circ_{e,e})$ defined by $x_{\circ_{e,e}} y = (x/e)(e\ldv y)$ has exponent $2$.
\end{corollary}
\begin{proof}
The principal loop isotope $(X,\circ_{e,e})$ has identity element $ee$. By Lemma \ref{Lm:BSRitp}, $x\circ_{e,e} x = ee$.
\end{proof}

\begin{corollary}\label{Cr:BS2group}
If a both-sided rumple $X$ is isotopic to a group, then it is isotopic to an elementary abelian $2$-group.
\end{corollary}
\begin{proof}
If a quasigroup $X$ is isotopic to a group $G$ then all loop isotopes of $X$ are isomorphic to $G$. We are done
by Corollary \ref{Cr:BSRitp}.
\end{proof}

\subsection{Generators of the displacement group}

In this section we prove that all generators of the displacement group $\dis{X}$ of a both-sided rumple $X$
have order dividing $4$.

\begin{lemma}\label{Lm:BSRaux}
Let $X$ be a both-sided rumple. Then:
\begin{enumerate}
\item $L_{xy}\inv L_{yy} = L_{yx}\inv L_{xx}$ for every $x$, $y\in X$.
\item $L_{xx} L_{yx}\inv L_{yy} = L_{xy}$ for every $x$, $y\in X$.
\end{enumerate}
\end{lemma}
\begin{proof}
(1) We have
\[
    L_{yy}
    = L_{yy} L_y L_y\inv
    \overset{\ref{Pr:BSRsigma}.4}{=} \sigma^2 L_y\inv
    \overset{\ref{Pr:BSRsigma}.2}{=} L_{\sigma^2(y)}\inv \sigma^2
    \overset{\ref{Pr:BSRsigma}.3}{=} L_{xy\cdot yx}\inv \sigma^2.
\]
Reversing the roles of $x$ and $y$, we obtain $L_{xx} = L_{yx\cdot xy}\inv \sigma^2$. It therefore remains to prove $L_{xy}\inv L_{xy\cdot yx}\inv = L_{yx}\inv L_{yx\cdot xy}\inv$, that is, $L_{xy\cdot yx} L_{xy} = L_{yx\cdot xy} L_{yx}$, which is a consequence of \eqref{Eq:LeftRumpTrans}.

(2) Let us first establish
\begin{equation}\label{Eq:99-temp}
    (x\ldv yx)x = yy\qquad\text{and}\qquad(x\ldv yx)^2 = xy.
\end{equation}
For the first identity, calculate
\[
    xx\cdot (x\ldv yx)x
    \overset{\eqref{Eq:RightRump}}{=} x(x\ldv yx)\cdot x(x\ldv yx)
    = yx\cdot yx
    \overset{\ref{Pr:BSRsigma}.2}{=} xx\cdot yy
\]
and cancel $xx$ on the left. For the second identity, observe
\[
    yx\cdot xy
    \overset{\ref{Pr:BSRsigma}.3}{=} \sigma^2(x)
    \overset{\ref{Pr:BSRsigma}.4}{=} x(x\ldv yx)\cdot (x\ldv yx)^2
    = yx\cdot (x\ldv yx)^2
\]
and cancel $yx$ on the left. Then
\begin{align*}
    L_{xx} L_{yx}\inv L_{yy}
    &\overset{\phantom{\eqref{Eq:RightRumpTrans}}}{=} L_{xx} L_x L_x\inv L_{yx}\inv L_{yy}
    \overset{\ref{Pr:BSRsigma}.4}{=} \sigma^2 L_x\inv L_{yx}\inv L_{yy}
    = \sigma^2 (L_{x(x\ldv yx)} L_x)\inv L_{yy} \\
    &\overset{\eqref{Eq:RightRumpTrans}}{=} \sigma^2 (L_{(x\ldv yx)x}L_{x\ldv yx})\inv L_{yy}
    \overset{\eqref{Eq:99-temp}}{=} \sigma^2 L_{x\ldv yx}\inv L_{yy}\inv L_{yy}
    = \sigma^2 L_{x\ldv yx}\inv
    \overset{\ref{Pr:BSRsigma}.4}{=} L_{(x\ldv yx)^2}
    \overset{\eqref{Eq:99-temp}}{=} L_{xy}\,.\qedhere
\end{align*}
\end{proof}

\begin{proposition}\label{Pr:exp4}
Let $X$ be a both-sided rumple. Then
\[
    (L_x L_y\inv )^4 = 1 = (R_x R_y\inv )^4
\]
for every $x$, $y\in X$.
\end{proposition}
\begin{proof}
We will prove the first equality. The second equality follows by a dual argument. We have
\begin{align*}
    (L_x L_y\inv)^2
    &\overset{\phantom{\ref{Pr:BSRsigma}.4}}{=} L_x L_y\inv \cdot L_x L_y\inv
    \overset{\eqref{Eq:LeftRumpTrans}}{=} L_{xy}\inv L_{yx} \cdot L_{xx}\inv L_{xx} L_x L_y\inv L_{yy}\inv L_{yy}\\
    &\overset{\ref{Pr:BSRsigma}.4}{=} L_{xy}\inv L_{yx} \cdot L_{xx}\inv \sigma^2 \sigma^{-2} L_{yy}
    \overset{\ref{Lm:BSRaux}.2}{=} L_{xy}\inv L_{yy} L_{xy}\inv L_{xx} L_{xx}\inv L_{yy}
    = L_{xy}\inv L_{yy} L_{xy}\inv L_{yy}.                                   &
\end{align*}
Thus we have
\begin{align*}
    (L_x L_y\inv)^4
    &\overset{\phantom{\ref{Lm:BSRaux}.1}}{=}L_{xy}\inv L_{yy} L_{xy}\inv L_{yy} L_x L_y\inv L_x L_y\inv
    \overset{\ref{Lm:BSRaux}.1}{=} L_{xy}\inv L_{yy} L_{yx}\inv L_{xx} L_x L_y\inv L_x L_y\inv\\
    &\overset{\ref{Pr:BSRsigma}.4}{=} L_{xy}\inv L_{yy} L_{yx}\inv L_{yy} L_y L_y\inv L_x L_y\inv
    \overset{\ref{Lm:BSRaux}.1}{=} L_{xy}\inv L_{yy} L_y (L_{yx} L_y)\inv L_{yy} L_x L_y\inv\\
    &\overset{\eqref{Eq:LeftRumpTrans}}{=} L_{xy}\inv L_{yy} L_y (L_{xy} L_x)\inv L_{yy} L_x L_y\inv
    \overset{\ref{Pr:BSRsigma}.4}{=} L_{xy}\inv L_{xx}L_{xy}\inv L_{yy}L_xL_y\inv \\
    &\overset{\ref{Lm:BSRaux}.1}{=} L_{xy}\inv L_{xx} L_{yx}\inv L_{xx} L_x L_y\inv
    = L_{xy}\inv L_{xx} L_{yx}\inv L_{yy} L_y L_y\inv
    \overset{\ref{Lm:BSRaux}.2}{=} L_{xy}\inv L_{xy}
    = 1\,.\qedhere
\end{align*}
\end{proof}

Although Proposition \ref{Pr:exp4} is interesting in its own right, it also has an implication
for loop isotopes of a both-sided rumple $X$: it turns out that the conclusion of the proposition
is equivalent to the assertion that every loop isotopic to $X$ is power-associative of exponent
dividing $4$.
(The proof is not difficult but it would take us a bit far afield of the main topic of this paper.) Combining this with Corollary \ref{Cr:BSRitp}, we can conclude that if $X$ is a both-sided rumple which is not isotopic to a group, then some loop isotope achieves exponent $4$. This is because if all loops isotopic to a given quasigroup have exponent $2$, then those loops are isomorphic abelian groups \cite{Bruck}. 

\end{document}